\documentclass{article}

\usepackage{amsfonts}
\usepackage{amsmath,amsthm,amssymb}
\usepackage{graphicx}
\usepackage{tikz-qtree}
\usepackage[caption=false]{subfig}
\usepackage{hyperref}
\usepackage{cleveref}
\usepackage{geometry}

\usepackage{epstopdf}
\usepackage{algorithmic}
\ifpdf
  \DeclareGraphicsExtensions{.eps,.pdf,.png,.jpg}
\else
  \DeclareGraphicsExtensions{.eps}
\fi

\usepackage{color}
\definecolor{red}{rgb}{1.0,0.,0.}
\newcommand{\todo}[1]{\textcolor{red}{[[#1]]}}

\newtheorem{remark}{Remark}
\newtheorem{lemma}{Lemma}
\newtheorem{theorem}{Theorem}
\newtheorem{corollary}{Corollary}
\newtheorem{definition}{Definition}

\title{A path integral method for solution of the wave equation with continuously-varying coefficients}

\author{Jithin D. George\thanks{Department of Engineering Sciences and Applied Mathematics, Northwestern University (\href{jithindgeorge93@gmail.com})}
\and David I. Ketcheson\thanks{Computer, Electrical, and Mathematical Sciences \& Engineering Division,
King Abdullah University of Science and Technology, 4700 KAUST, Thuwal
23955, Saudi Arabia. (\href{david.ketcheson@kaust.edu.sa})}
\and Randall J. LeVeque\thanks{Department of Applied Mathematics, University of
Washington, Seattle, WA 98195-3925. 
(\href{rjl@uw.edu})}
}

\usepackage{amsopn}

\DeclareMathOperator{\sech}{sech}
\DeclareMathOperator{\Vol}{Vol}

\newcommand{\ignore}[1]{}
\newcommand{\eqn}[1]{(\ref{#1})}

\newcommand{\bx}{{\mathbf x}}
\newenvironment{mat}{\left[ \begin{array}{ccccccccccccc}}{\end{array}\right]}
\newenvironment{rmat}{\left[ \begin{array}{rrrrrrrrrrrrr}}{\end{array}\right]}
\newcommand\bcm{\begin{mat}}
\newcommand\ecm{\end{mat}}
\newcommand\brm{\begin{rmat}}
\newcommand\erm{\end{rmat}}

\newcommand\Paths{\mathcal P}
\newcommand\traveltime{\tau}

\newcommand{\Xset}{{\mathcal X}}
\newcommand{\Bset}{{\mathcal B}}
\newcommand{\Real}{{\mathbb R}}

\newcommand{\Ct}{C_T}
\newcommand{\Cr}{C_R}

\newcommand{\lef}{-}
\newcommand{\righ}{+}
\newcommand{\wlim}{\overline{w}}
\newcommand{\plim}{\overline{p}}
\newcommand\new[1]{#1}

\begin{document}

\maketitle

\begin{abstract}
A new method of solution is proposed for solution of the wave equation
in one space dimension \new{with continuously-varying coefficients.
By considering all paths along which information arrives at a given
point, the solution is expressed as an infinite series of integrals,
where the integrand involves only the initial data and the PDE coefficients.
Each term in the series represents the influence of paths with a fixed
number of turning points.}
We prove that the series
converges and provide bounds for the truncation error.
The effectiveness of the approximation is illustrated with examples.
We illustrate an interesting combinatorial connection between the traditional
reflection and transmission coefficients for a sharp interface, and Green's coefficient
for transmission through a smoothly-varying region.  \end{abstract}

\section{Introduction and physical setting}
\label{sec:intro}
We consider the Cauchy problem for the linear one-dimensional wave equation
\begin{align}
    u_{tt} & = \frac{1}{\rho(x)}\left(K(x) u_x(x,t)\right)_x,
\end{align}
which can also be written in first-order form as
\begin{equation}\label{acoustics}
\begin{split}
    p_t(x,t) + K(x) u_x(x,t) & = 0 \\
    u_t(x,t) + \frac{1}{\rho(x)}p_x(x,t) & = 0.
\end{split}
\end{equation}
Here we have used the notation of acoustics: $p$ is pressure, $u$ is velocity,
$K$ is the bulk modulus, and $\rho$ is the density.
Linear wave equations with the same mathematical structure arise in many other
applications, with different interpretations of the material parameters, such
as elasticity, electromagnetics, and linearized fluid dynamics or water waves.
If the coefficients $(\rho(x), K(x))$ are constant or piecewise-constant,
the problem may be solved exactly by the method of characteristics.
On the other hand, for more general functions $\rho(x)$ and/or $K(x)$
the method of characteristics does not substantially simplify the problem
since the solution varies at every point along a characteristic in a way
that is coupled to other characteristics.

In this work we propose and demonstrate a method for approximately solving the general
Cauchy problem for \eqref{acoustics} in the presence of arbitrary variation
in $\rho$ and $K$, by grouping characteristic paths according to the number of reflections.
Our interest originated in a study of the shoaling of water waves over a continental
shelf and our complementary works \cite{JGeorgeMS,shoalingpaper} contain
more discussion of this application and several illustrative examples
using the linearized shallow water equations, a special case that is
also discussed in Remark \ref{remark:sw} below.
\new{Code to reproduce the numerical experiments in this paper is available online.\footnote{\url{https://github.com/ketch/characteristics_rr}}}

\new{We focus on scattering of a localized pulse or front in a medium
including a region of continuous variation.
In this setting the standard analytical tool, the method of characteristics,
leads to a system of ordinary differential equations \cite{lax2006hyperbolic}.
In practice, to compute the solution at a single point from this expression, one
must discretize the infinite system of ODEs and then solve them numerically.
Herein we derive an expression for the solution at any point in terms of
just a multidimensional integral.
In practice this gives a method that is much more analytical
than the classical method of characteristics, since the resulting expression
only requires the evaluation of integrals (rather than the solution
of an infinite system of ODEs).  Our goal is not to provide an efficient computational
method for \eqref{acoustics}, but rather a semi-analytical tool that may provide
insight into solutions.

Among the vast literature on solutions of the wave equation in heterogeneous
media there are other approaches that bear some relation to ours.  In most cases,
a piecewise-constant approximation is used; a
combinatorial solution for scattering from arbitrary piecewise-constant media
was developed in \cite{gibson2014combinatorics,gibson2019disk}. For an approach similar to
ours but in the setting of time-harmonic solutions, see
\cite{bremmer1951wkb,landauer1951reflections,schelkunoff1951remarks}.

The method developed in this paper results in an approximation
series that has some similarity with the Born scattering series
used in seismic imaging \cite{weglein2003inverse,innanen2009born}.
Each term of the Born series arises as a perturbation expansion of
the Green's function solution depending on transmission or reflection
at any point, while in our series each term arises solely from
paths involving a particular number of turning points in a
heterogeneous medium. The goal of forward seismic imaging is often
to get a primary approximation of waves that have only suffered one
reflection or to remove the effect of multiply reflected waves. So,
methods developed to approximate primary waves such as in
\cite{innanen2008direct} and multiple removal algorithms like that
explored in \cite{innanen2009born} bear some visual and physical
similarities with the integrals developed in this paper.}

In the remainder of this section we briefly review the mathematics of characteristics and reflection
in one dimension.

\subsection{The method of characteristics: homogeneous media}
Defining $q = [p,u]^T$, the system \eqref{acoustics} can be written as $q_t + A(x) q_x = 0$,
where $A$ has the eigenvalue decomposition $A=V(x)\Lambda V^{-1}(x)$ with
\begin{align}
    V(x) & = \begin{bmatrix} 1 & 1 \\ \frac{-1}{Z(x)} & \frac{1}{Z(x)} \end{bmatrix} &
    \Lambda(x) & = \begin{bmatrix} -c(x) & 0 \\ 0 & c(x) \end{bmatrix}.
\end{align}
Here $Z(x)=\sqrt{K(x)\rho(x)}$ is known as the impedance and $c(x)=\sqrt{K(x)/\rho(x)}$
is the sound speed.
If $K(x)$ and $\rho(x)$ are constant (or more generally, if $Z(x)$ is constant)
then $V(x)$ is also constant and, setting $w(x,t)=V^{-1}q(x,t)$,
\eqref{acoustics} can be rewritten as \begin{equation} \label{advection}
    w_t + \Lambda(x) w_x = 0.
\end{equation}
System \eqref{advection} consists of two decoupled advection equations,
indicating that one component of the solution ($w_1$) travels to the left
(with velocity $-c$) while the other ($w_2$) travels to the right (with velocity $+c$).
Lines of constant $x+ct$ and $x-ct$ are referred to as characteristics.
The solution is simply the sum of the components transmitted along the
two characteristic families:
\begin{align} \label{simple-characteristics-solution}
    p(x,t) = w_1(x+ct,0) + w_2(x-ct,0).
\end{align}
%

\subsection{Piecewise-constant media: reflection and transmission}\label{subsection:piecewise}
The method of characteristics can also be used to find the exact solution of
\eqref{acoustics} if $K(x)$ and $\rho(x)$ are piecewise-constant functions.  Within each
constant-coefficient domain the characteristic velocities are $\pm c(x)$.
Consider a single interface where
the impedance jumps from $Z_-$ on the left to $Z_+$ on the right.
Let $v_1^\pm, v_2^\pm$ denote the respective columns of $V(0^\pm)$.
For an incident right-going wave, the incident ($p_0$),
transmitted ($p_T$), and reflected ($p_R$) wave pressures are related by
\begin{align} \label{transreflsys}
    p_0 v_2^+ & = p_T v_2^+ + p_R v_1^-.
\end{align}
Solving system \eqref{transreflsys} reveals that the transmitted and reflected
waves are related to the incident wave by the transmission and reflection
coefficients:
\begin{subequations} \label{discontTR}
\begin{align} \label{Ctrans}
    \Ct(Z_-,Z_+) & := \frac{p_T}{p_0} =  \frac{2Z_+}{Z_- + Z_+}, \\
    \Cr(Z_-,Z_+) & := \frac{p_R}{p_0} = \frac{ Z_+-Z_-}{Z_-+ Z_+}. \label{Crefl}
\end{align}
\end{subequations}

\subsection{Smoothly-varying media}
Wherever the impedance $Z(x)$ is not constant, the system \eqref{acoustics} cannot
be decoupled as in \eqref{advection} because the matrix $V(x)$ that relates $q$ and $w$
varies in space.  If $Z(x)$ is differentiable, we have
$w_x = (V(x)^{-1}q)_x=V^{-1}(x)q_x + (V^{-1}(x))'q$ and we obtain
instead of \eqref{advection} the system
\begin{equation}\label{advection-coupled}
\begin{split}
    w_t + \Lambda(x) w_x & = (V^{-1})' q \\
                         & = (V^{-1})' V w.
\end{split}
\end{equation}
Here
\begin{align}
(V^{-1}(x))'V(x) & = \frac{1}{2} \frac{Z'(x)}{Z(x)} \left[ \begin{array}{rr} 1 & -1 \\ -1 & 1 \end{array} \right].
\end{align}
\new{Since the left-hand side of \eqref{advection-coupled} is decoupled, we can write
\eqref{advection-coupled} as a system of ODEs using a simple coordinate transformation,
and the solution of this system can be shown to be that of the original PDE
(see e.g. \cite[Section 4.1]{lax2006hyperbolic}).  Using this approach, to find
the solution $p(x,t)$ at a single point still requires solving an infinite number of
ODEs.  In this work we derive a semi-analytical method that only requires computing
an iterated integral for each point value of the solution.}

We see from \eqref{advection-coupled} that information is still transmitted
along characteristics, but the amplitude of each component is modified by the source
terms that couple the characteristic variables through reflection.
The coefficient
\begin{align} \label{inf-refl-coeff}
    r(x) = \frac{Z'(x)}{2Z(x)}
\end{align}
gives the amplitude of these reflections and we refer to it
as the {\em infinitesimal reflection coefficient}.

The infinitesimal reflection ceofficient $r(x)$ is related to the traditional
reflection coefficient $R(Z_-,Z_+)$; if we take $Z(x)$ to be a continuous
function with value $Z_-$ at $x$ and  value $Z_+$ at $x+\Delta x$, the ratio
$R/\Delta x$ approaches $r(x)$ as $\Delta x$ tends to zero:
\begin{equation}
\frac{R}{\Delta x} \approx \frac{1}{\Delta x}\frac{Z(x+\Delta x) -
Z(x)}{Z(x+\Delta x) +
Z(x)} \approx \frac 1 2 \frac{Z'(x)}{Z(x)}.
\end{equation}


\begin{figure}
\begin{centering}
\subfloat[Homogeneous medium]{ \includegraphics[width=2in]{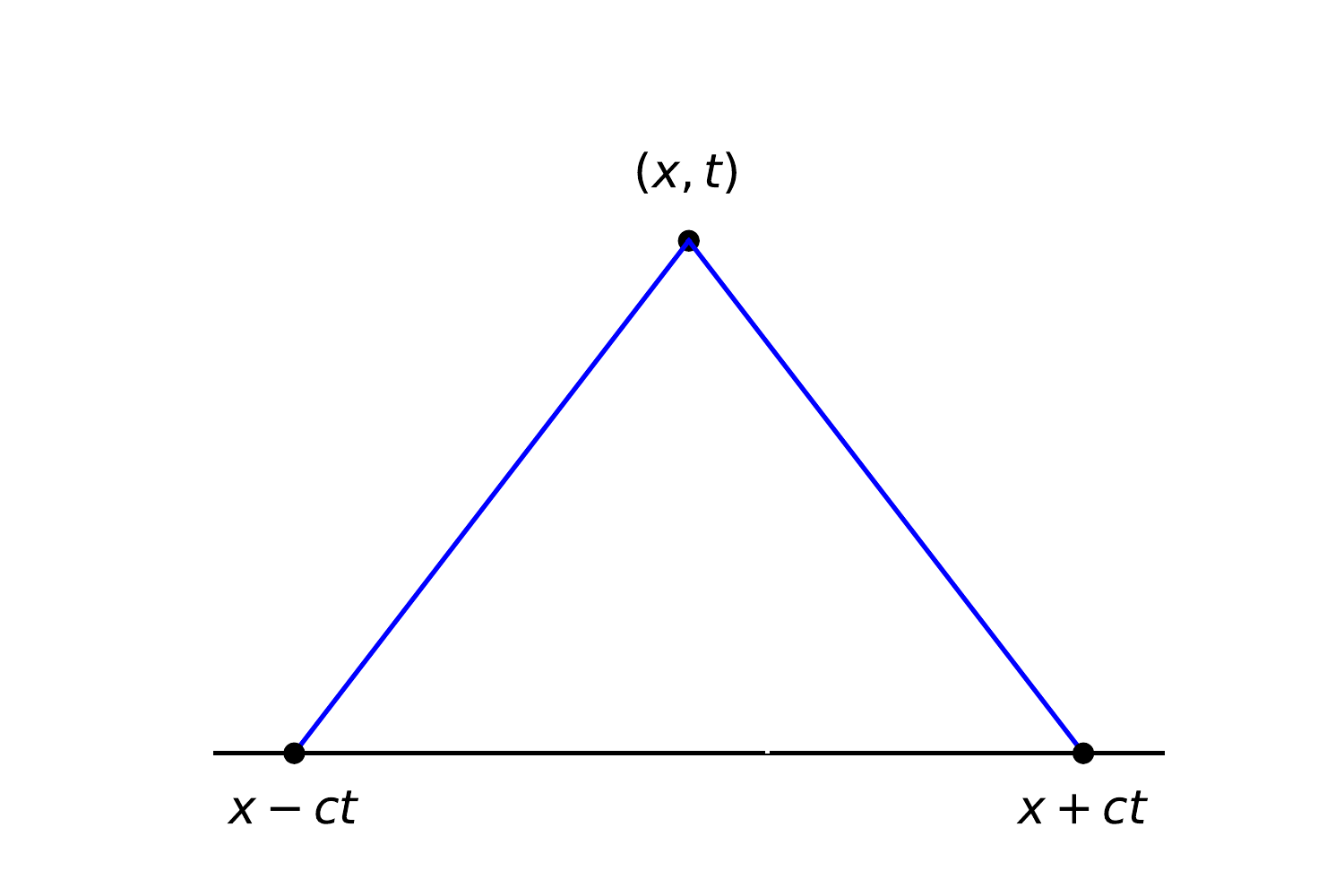}}
\subfloat[Piecewise-constant medium]{ \includegraphics[width=2in]{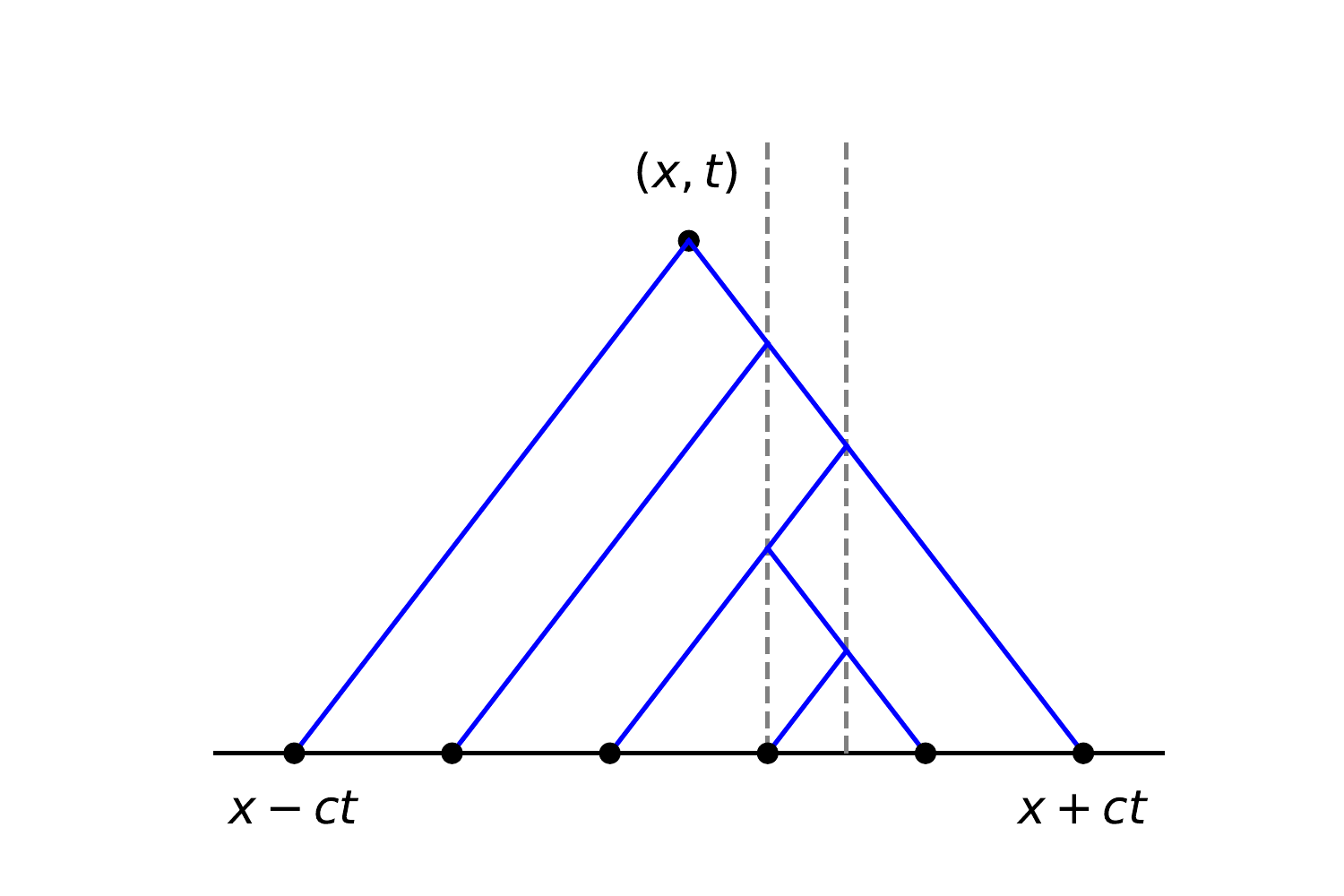}}
\subfloat[Continuously-varying medium]{ \includegraphics[width=2in]{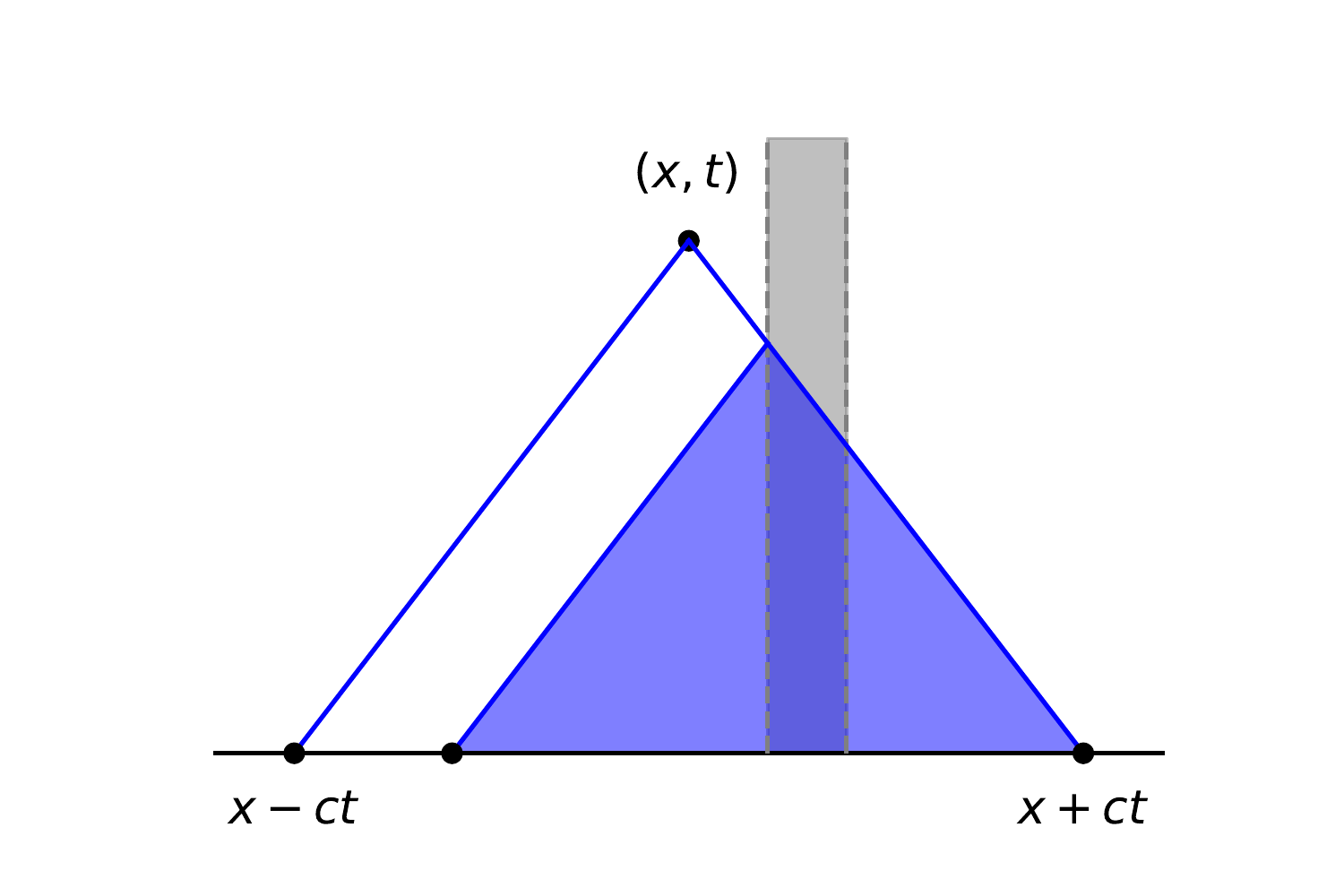}}
\caption{Characteristics in three different types of media.  In the homogeneous medium,
the solution at each point is determined by just two characteristics.  In the piecewise-constant
medium (with material interfaces indicated by dashed lines), the solution at
each point is determined by a finite number of characteristics.
In the continuously-varying medium (with $Z(x)$ varying throughout the grey-shaded region), the
solution at the indicated point depends on all characteristics within the blue-shaded region.
\label{fig:media}}
\end{centering}
\end{figure}

Characteristics for each of the three classes of media just discussed are illustrated
in Figure \ref{fig:media}.  We see that
in the presence of constant or piecewise-constant impedance, the number of
characteristics that must be accounted for to compute the solution at a given
point is finite.  On the other hand, if
$Z(x)$ varies continuously then there are in general infinitely
many characteristics that influence a given point.
The technique developed in the rest of this work is based on the hypothesis
that the dominant contributions to the solution come from accounting for
paths with relatively few reflections.  Here a path is a continuous, piecewise
smooth curve in the $x-t$ plane such that each smooth part follows a
characteristic and $Z'(x)\ne0$ at each point of non-smoothness.  This hypothesis
is clearly reasonable when $|r(x)|<1$, since then each reflection must
diminish the significance of the corresponding characteristic path.
The motivation for this hypothesis more generally is given in Section
\ref{sec:convergence}.

\section{Characteristics in continuously-varying media}
In this section we develop an approximate solution to \eqref{acoustics}
in the form of an infinite series.
We focus on the case of a finite region of variation in the spatial
coefficients, as illustrated in Figure \ref{fig:setting}:
\begin{align} \label{coefficients}
    (K(x),\rho(x)) & = \begin{cases} (K_\lef, \rho_\lef) &  x < 0  \\
                                     (K(x), \rho(x)) & 0 \le x \le x_\righ \\
                                     (K_\righ, \rho_\righ) & x > x_\righ. \end{cases}
\end{align}
Here $x_\righ$ is the width of the region of varying coefficients, and need not be small.
For simplicity we consider the case of a right-going disturbance that is initially
confined to $x<0$, and investigate the resulting reflected and transmitted disturbances.
Thus
\begin{align} \label{initial-data}
    \begin{bmatrix} p(x,0) \\ u(x,0) \end{bmatrix} & = \begin{cases} p_0(x) \begin{bmatrix} 1 \\ 1/Z_- \end{bmatrix} & x < 0 \\ 0 & x \ge 0. \end{cases}
\end{align}
We assume for simplicity that $Z(x)$ is continuous.
Our method and results can be generalized in a natural way to arbitrary initial data and
piecewise continuous media.
\begin{figure}
\begin{centering}
\includegraphics[width=4in]{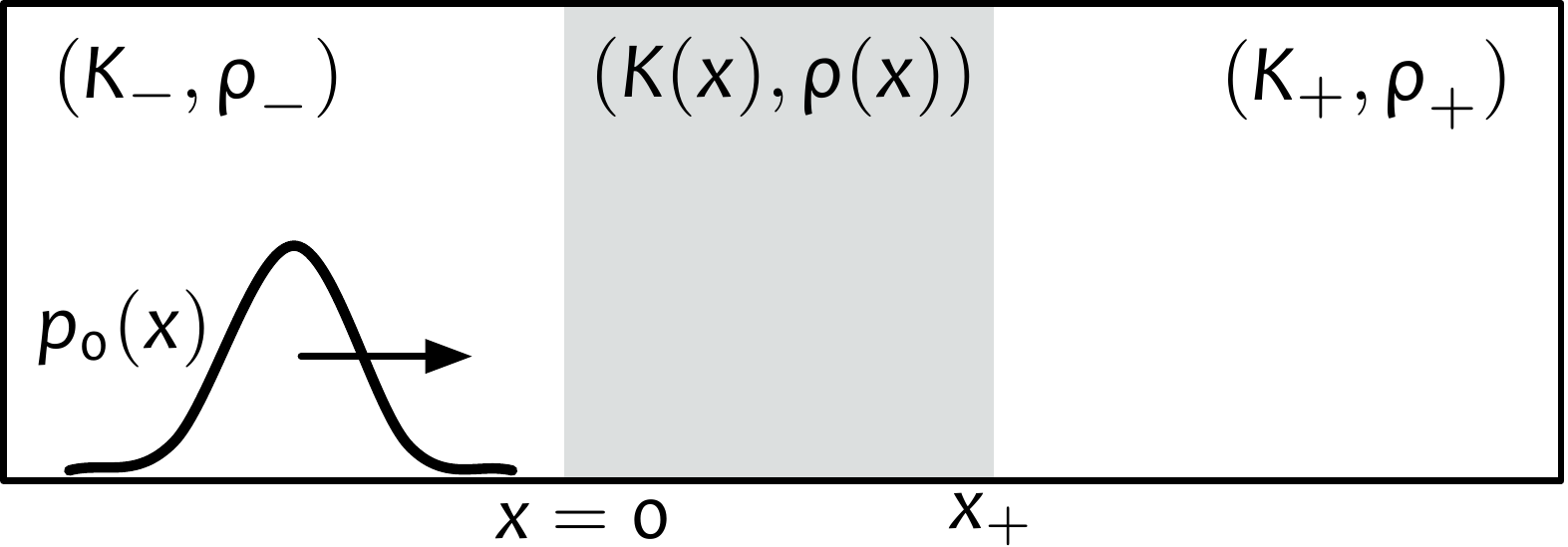}
\caption{The setting for most of the paper.
\label{fig:setting}}
\end{centering}
\end{figure}

Outside of the region $[0,x_\righ]$, characteristics are straight lines in the
$x$--$t$ plane.
Let $X(t)$ denote the characteristic starting from $x=0$ at time zero;
i.e., the solution of the initial value ODE
\begin{align} \label{charODE}
    X'(t) & = c(X(t)) & X(0) & = 0 & t \in [0,t_\righ].
\end{align}
Here $t_\righ$ is the crossing time so that $X(t_\righ)=x_\righ$.
It is convenient in what follows to extend $X(t)$ by defining
$X(t) = 0$ for $t<0$ and $X(t)=x_\righ$ for $t>t_\righ$.

\subsection{Amplification or attenuation along characteristics: Green’s Law}
In general the pressure is given by $p=w_1 + w_2$;
for the case of a pure right-going pulse \eqref{initial-data}, for which $w_1$ is zero,
we have $p(x,0)=w_2(x,0) = p_0$.
According to \eqref{advection-coupled}, along the path $X(t)$ the value of $w_2$ (and
hence the value of $p$) satisfies the ODE
\begin{align} \label{eq:attenuation}
    p'(X(t)) & = \frac{Z'(X(t))}{2Z(X(t))} p(X(t))
\end{align}
with solution
\begin{align}\label{Greensac}
    p(X) & = \left(\frac{Z(X)}{Z_\lef}\right)^{1/2} p_0.
\end{align}
In particular, at $x=x_\righ$ we have
\begin{equation}\label{green23}
\frac{p_\righ}{p_0} = \left(\frac{Z_\righ}{Z_\lef}\right)^{1/2}  = C_G.
\end{equation}
Thus the amplitude of the unreflected part of the wave (for $x\ge x_\righ$) is
$C_G p_0$ for any smoothly varying $Z(x)$, and depends only on the values
$Z_\lef$ and $Z_\righ$; it is independent of how $Z$ varies over $[0,x_\righ]$.
As we will see, \eqref{green23} represents the first term in an infinite
series that sums to the transmission coefficient $\Ct$.

\begin{remark} \label{remark:sw}
We use $C_G$ for the quantity defined in \eqn{green23} since this is the
amplification factor given by Green's law in the context of shoaling,
as we discuss in more detail in \cite{shoalingpaper}.
The linearized shallow water equations used there can
be put in the form \eqn{acoustics} by introducing $p(x,t)$ as the depth
perturbation of a small amplitude long wave on a background water depth
$h(x)$, and $\mu(x,t)$ as the momentum perturbation. Then the linearized
shallow water equations can be written in the nonconservative form
\begin{equation}\label{linswe}
\begin{split}
    p_t(x,t) + \mu_x(x,t) & = 0 \\
    \mu_t(x,t) + gh(x) p_x(x,t) & = 0,
\end{split}
\end{equation}
where $g$ is the gravitational constant.
This differs from the conservative form used in \cite{shoalingpaper}, and
has the same form as \eqn{acoustics} if we set $K(x)\equiv 1$ and
$\rho(x) = 1/(gh(x))$.  Then the wave speed is $c(x) = \sqrt{gh(x)}$,
the impedance is $Z(x) = 1/\sqrt{gh(x)}$,
and $C_G = (h_\lef/h_\righ)^{1/4}$. This is the standard form of
Green's law used to estimate the amplification of a shoaling wave
as it passes into shallower water, in which case $h_\lef > h_\righ$.
Note that this particular application is a special case in that there
is only a single variable coefficient $h(x)$, so it is not possible
to vary the wave speed and impedance separately.
\end{remark}

\begin{figure}
\begin{centering}
\includegraphics[width=4in]{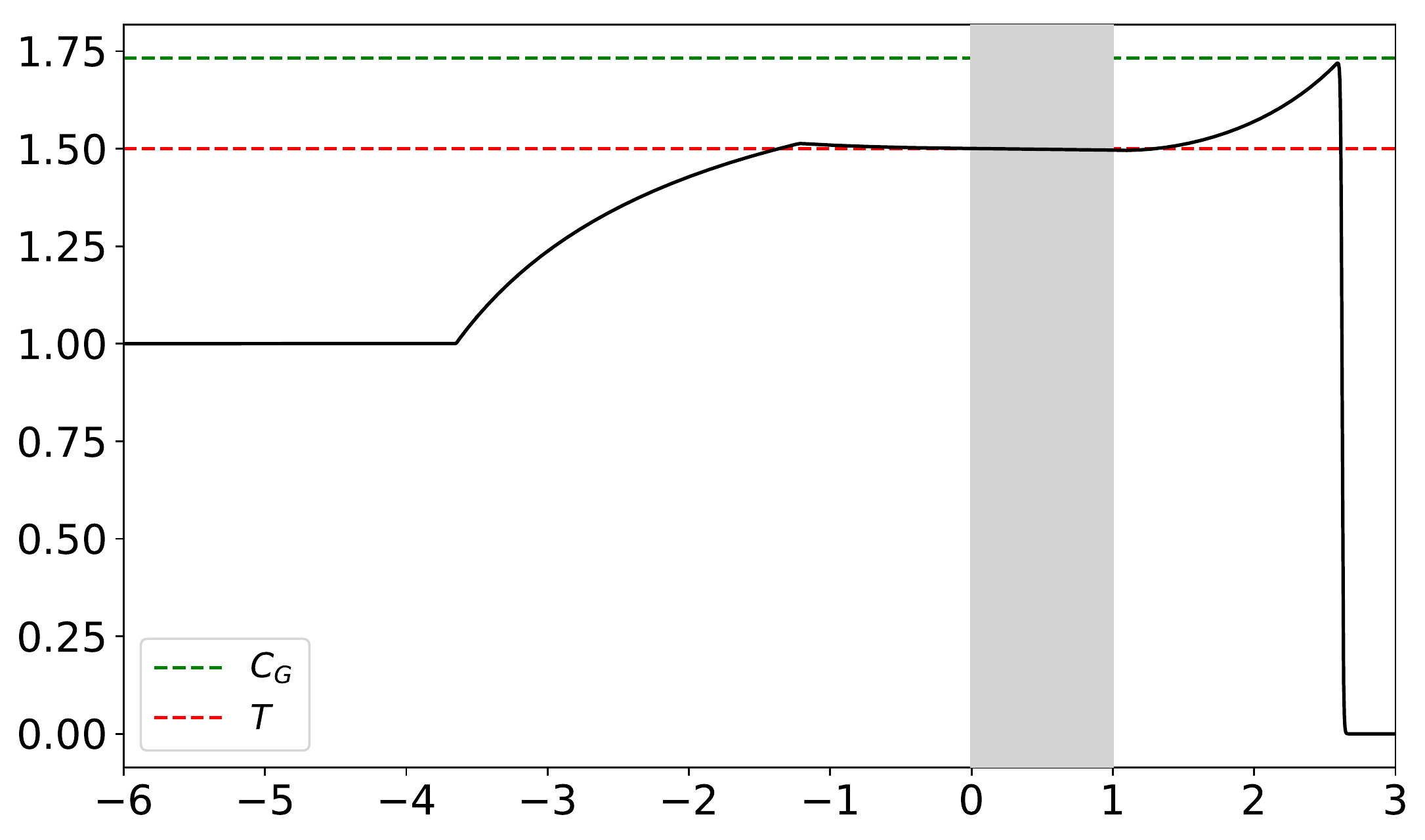}
\caption{Transmission and reflection of an initial right-going step.  The leading edge of the transmitted
part has amplitude $C_G$ (the amplitude in the absence of reflections), then tends to $\Ct$ at later
times as multiply reflected components contribute.
\label{fig:Greenslaw}}
\end{centering}
\end{figure}

Both the amplification factor $C_G$ and the transmission coefficent $\Ct$ defined
in \eqref{Ctrans} are related to the amplitude of transmitted waves.  Their differing
roles are illustrated in Figure \ref{fig:Greenslaw}, where we consider the propagation of
a step function (taking $p_0(x)=1$, and with an impedance that grows linearly from $Z_-=1$ to $Z_+=3$
in the region $[0,1]$).  Since $C_G$ governs the amplification along
characteristics, the leading part of the transmitted wave (which is unaffected
by paths with turning points, since they will emerge at later times)
has amplitude $C_G$.  Meanwhile, $\Ct$ accounts for the cumulative
effect of all paths (including those that have turning points)
and so the amplitude of the transmitted wave at long times approaches $\Ct$.

In Figure \ref{fig:limit} we consider what happens as $x_+$ tends to zero, for
fixed values of $Z_-, Z_+$.
We again take a step function as the initial condition (plotted as a dashed
line).  In this case the solution
is invariant if $x_+$ and $t$ are scaled by the same factor, but in Figure \ref{fig:limit}
we have plotted solutions for different values of $x_+$ all at the same time
$t$.  We see that as the region $[0,x_+]$
shrinks, the width of the transmitted peak becomes increasingly narrow until in
the limit $x_+=0$ (for which the impedance is discontinuous) the peak is gone
and we have a single intermediate state dictated by the transmission coefficient.
One way to think about this is that since the discontinuity
has infinitesimally small width, the effects of all relevant paths
must arrive in infinitesimally short time. The combinatorial relation between
the transmission/reflection coefficients and $C_G$ is further explored in
Section \ref{Asymptoticzigzags}.

\begin{figure}
\begin{centering}
\includegraphics[width=6in]{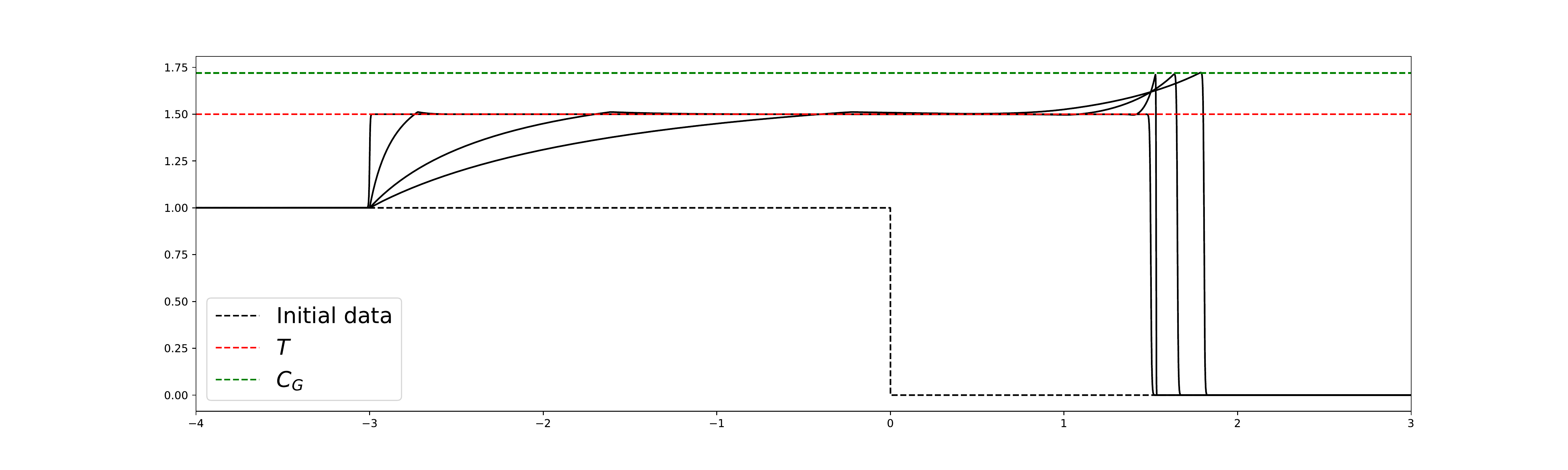}
\caption{A sequence of solutions with differing values of $x_+$ (the width of
the variable region).  In all cases the initial data is a step
function with unit amplitude (dashed line) and the impedance increases linearly from $Z_-=1$ to $Z_+=3$.
The width of the region of varying impedance is taken to be $[1, 1/2, 1/10, 0]$.
\label{fig:limit}}
\end{centering}
\end{figure}

\subsection{Approximating the reflected wave}
Let us now turn our attention to the reflected wave in Figure \ref{fig:Greenslaw}.
The main contribution to this wave comes from paths with one turning point,
as illustrated in Figure \ref{fig:Charsreflectedonce}.
The figure on the left shows two paths that emerge at $x=0$
at the same time but started from different initial points and were
reflected at different points.  It is evident that at any time $t$ there
will be such a path reaching $x=0$ that was reflected from point
$x$ for each $x\in(0,X(t/2))$, since the path reflected from the rightmost
point (the red path in the figure) must have traveled from $x=0$
to the point of reflection in time $t/2$.  For the initial condition $p_0(x)=1$,
the solution along each of these paths has the same initial amplitude.
In this case the combined amplitude of these reflected waves is
\begin{align}\label{once_reflected}
\int_0^{X(t/2)} r(x_1) d x_1 = \frac{1}{2}\log \left(\frac{Z(X(t/2))}{Z_-}\right),
\end{align}
where $r(x)$ is defined in \eqref{inf-refl-coeff}.
Figure \ref{fig:Charsreflectedonce} shows this diagrammatically. Initially, the
reflected wave only contains the contribution of paths reflected near
$x=0$.  After some time $2t_+$, reflections from the whole interval $[0,x_+]$
contribute, resulting in a constant asymptotic reflected amplitude.

\begin{figure}
\begin{centering}
\subfloat{ \includegraphics[width=2.8in]{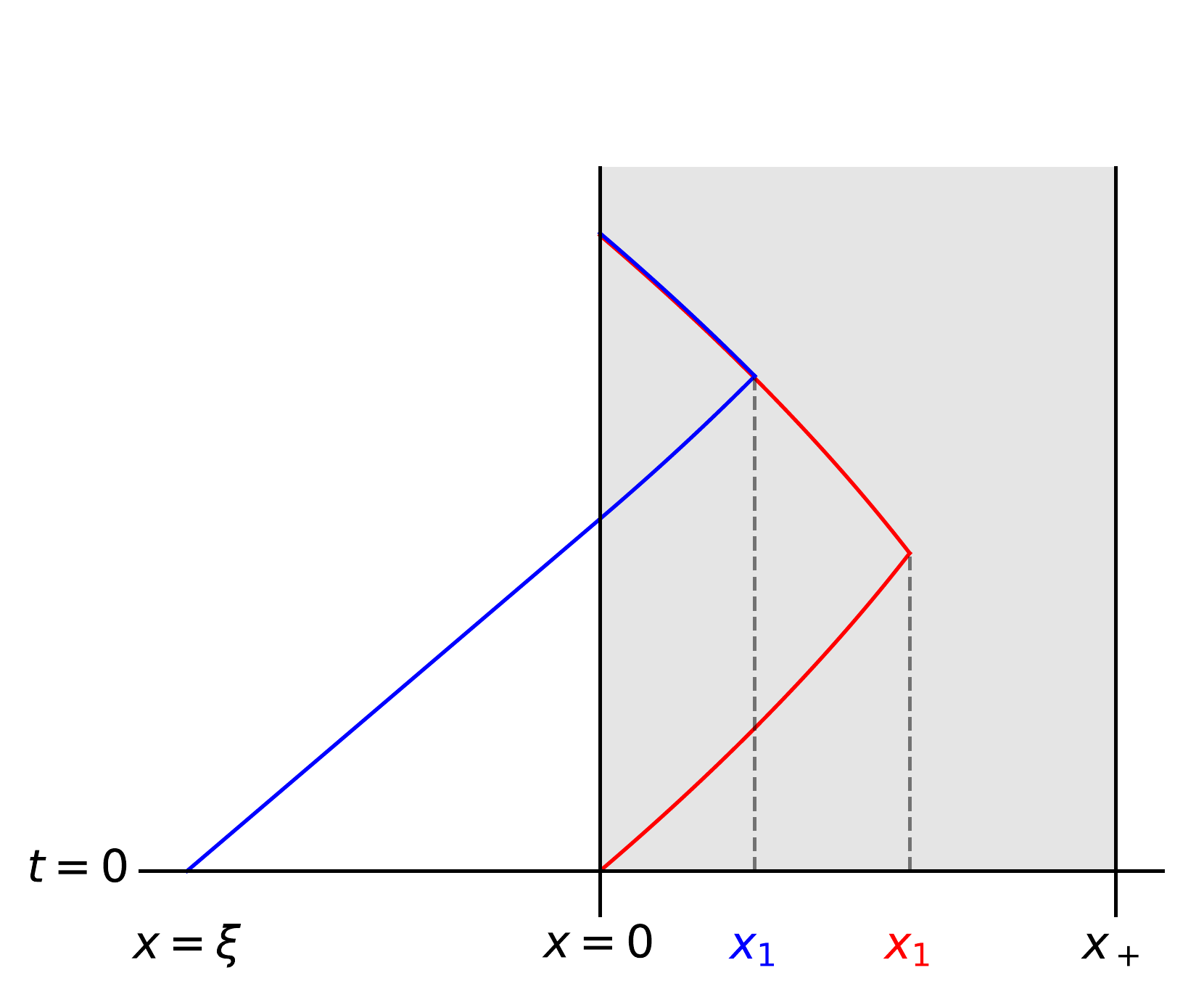}}
\subfloat{ \includegraphics[width=3in]{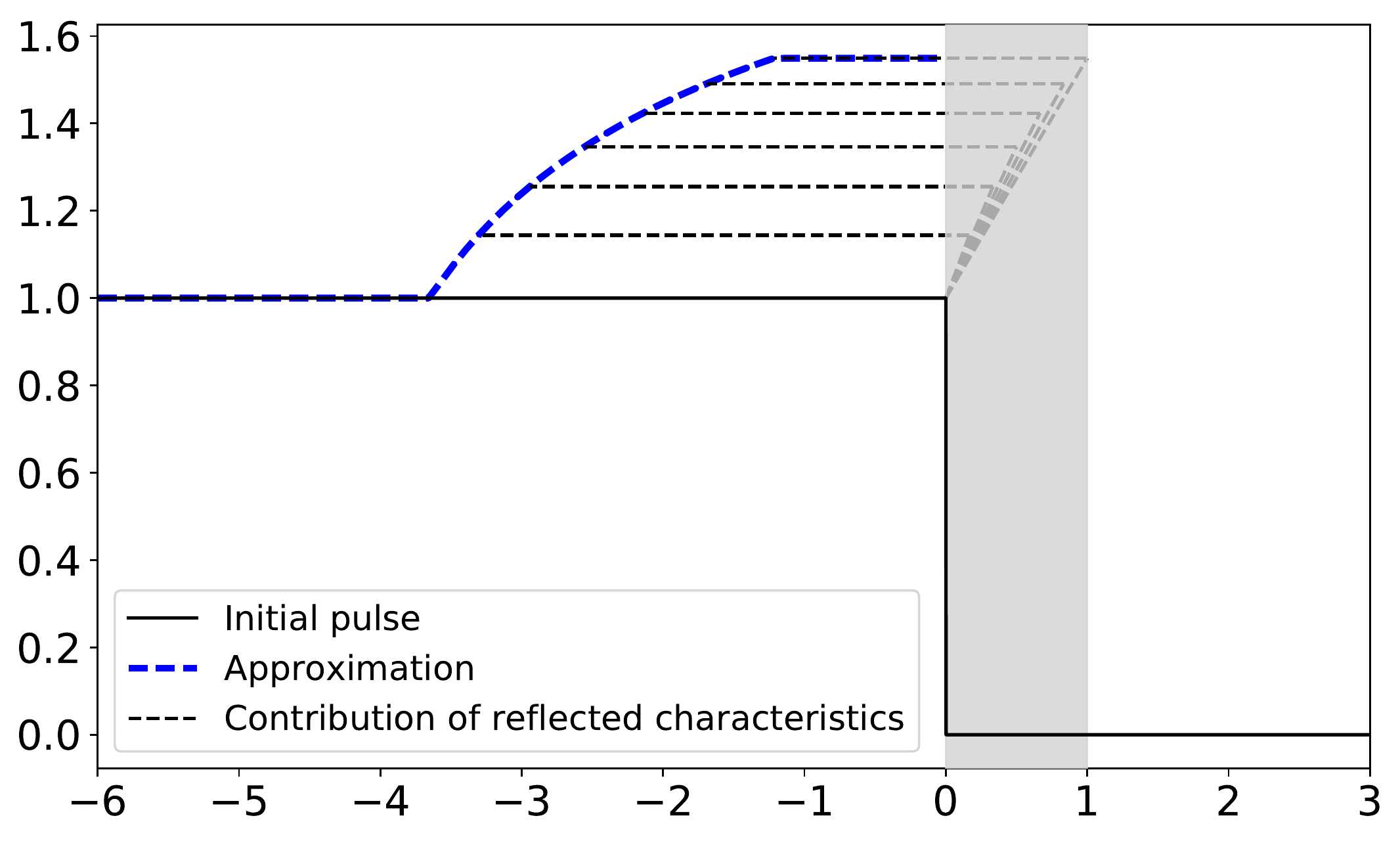}}
\caption{The reflected wave exists because of the cumulative effect of
paths with turning points in $[0,x_+]$.  This figure only shows paths with one
turning point, which is why the blue curve is an approximation to the reflected
wave.
\label{fig:Charsreflectedonce}}
\end{centering}
\end{figure}

\section{General solution by integrating over all paths}
In this section we construct a series for the solution at the boundaries of the
variable region:
\begin{subequations} \label{solution-series}
\begin{align}
    p(0,t) - p(-c_- t,0) & = \sum_{m=0}^\infty R_{2m+1}(t) \\
    p(x_+,t) & = \sum_{m=0}^\infty T_{2m}(t).
\end{align}
\end{subequations}
Here $R_n$ and $T_n$ denote contributions from paths involving
$n$ reflections.  We have effectively computed $T_0(t)$ and $R_1(t)$ already in
the previous sections; from our derivations of \eqref{green23} and \eqref{once_reflected} it
is straightforward to obtain the more general expressions
\begin{align*}
    T_0(t) & = C_G p_0(-c_- (t-t_+)) \\
    R_1(t) & = \int_0^{X(t/2)} p_0(-c_-(t-2\tau_1)) r(x_1)dx_1
\end{align*}
which give the part of the transmitted solution due to paths with no reflections
and the part of the reflected solution due to paths with a single
reflection, respectively.  Here and below, $\tau_j$ denotes the time for a
characteristic to reach $x_j$ from $x=0$.

The function $X(t)$ defined in \eqref{charODE} gives a characteristic; i.e., a path with
no points of reflection.  More generally, consider a path involving
reflection at the sequence of points $\bx = \{x_1, x_2, \dots, x_n\} \in [0,x_\righ]$,
which we refer to as the reflection point sequence for this path.
This path is a union of curves $X_j(t)$ ($j=0,1,2,\dots$), each of which is
the solution of an initial value problem:
\begin{align*}
    X_j'(t) & = (-1)^j c(X_j(t)) & X_j(t_j) & = x_j & t \in [t_j, t_{j+1}].
\end{align*}
Here $x_0=0$ and $x_{n+1}$ is either zero (for reflected paths) or
$x_\righ$ (for transmitted paths).  The value of $t_j$ is the time at which the path
reaches $x_j$.  For a given medium,
a path is determined completely by the reflection points $\bx$ and the initial
time $t_0$.  Some examples of such paths are given in Figures
\ref{fig:reflected} and \ref{fig:transmitted}.  Notice that the shape of the
curves $X_j$ depends on the variation of $c(x)$, but all can be obtained by
applying a temporal offset to $X(t)$ and (for left-going segments)
reflecting the curve $X(t)$ vertically in the $x$--$t$ plane.


In keeping with the method of characteristics, we would like to add up the
contributions of all paths arriving at a given place and time $(x,t)$.
One way to do this is to sum over all valid reflection point sequences.
Notice that the reflection point sequence cannot be an arbitrary
sequence of points in $[0,x_\righ]$.
We need to sum over all paths with an {\em alternating sequence} of
reflection points, as defined by:
\begin{definition}
A sequence $\bx = \{x_1, x_2, \cdots x_n\}$ is an {\em alternating (down-up) sequence} if
\begin{align*}
    x_j & \le x_{j-1} & \text{ for $j$ even, and} \\
    x_j & \ge x_{j-1} & \text{ for $j$ odd.}
\end{align*}
\end{definition}
Henceforth we use the term {\em alternating} to mean, specifically, down-up sequences.
Let
\begin{align} \label{paths-def}
    \Paths_n^{[\alpha,\beta]} := \{\bx \in [\alpha,\beta]^n : \bx \text{ is an alternating sequence.}\}.
\end{align}
Then an integral over all paths with $n$ reflection points in $[0,x_+]$ is
an integral over $\Paths_n^{[0,x_+]}$.  Thus the terms in \eqref{solution-series}
are given by the following iterated integrals. Note that $x_1$ in the outermost
integral can be anywhere in $[0, x+]$; then $x_2$ must be chosen in $[0, x_1]$, and
$x_3$ in $[x_2, x+]$, etc.
\begin{subequations} \label{TandR-integrals}
\begin{align} \label{reflection-integral}
R_{2m+1}(t) & := (-1)^m     \int \cdots \int_{\Paths^{[0,x_+]}_{2m+1}} p_0(\xi_R(\bx,0,t)) \prod_{j=1}^{2m+1} r(x_j) dx_j \\
T_{2m}(t)   & := (-1)^m C_G \int \cdots \int_{\Paths^{[0,x_+]}_{2m}}   p_0(\xi_T(\bx,x_+,t)) \prod_{j=1}^{2m}   r(x_j) dx_j. \label{transmission-integral}
\end{align}
\end{subequations}
Here $\xi_R(\bx,x,t)$ is the starting point for the path with reflection
points $\bx$ arriving eventually at $(x,t)$, while $\xi_T(\bx,x,t)$ is the starting
point for the path with reflection points $\bx$ arriving eventually
at $(x,t)$.
The factor $(-1)^m$ appears because the reflection coefficient for a characteristic
initially going left is $-r(x)$, so every even-numbered reflection involves
a factor of $-1$.
Thus $p_0(\xi)$ gives the initial solution value corresponding
to a given path, and the product of reflection coefficients
gives the part of that value that eventually contributes to the reflected or transmitted wave.
The limits of integration take into account that the reflection points
must be an alternating sequence.

The factor $C_G$ appearing in \eqref{transmission-integral} is due to variation
along a characteristic as described by the solution of \eqref{eq:attenuation}.
It is absent in \eqref{reflection-integral}
because the value of the solution along a characteristic traveling left changes
by exactly the reciprocal factor and so there is no net change in amplitude
for a path that returns to $x=0$.

We can compute the full solution (to any desired accuracy) by considering the
contributions from all paths involving $n=1, 2, \dots, N$
reflection points.  To complete this approach we only need to determine how $\xi$
depends on $\bx$ and $t$, which we do in the next two subsections.

\subsection{Reflection}

\begin{figure}
\begin{centering}
\subfloat[Characteristics contributing to $T_2(t)$\label{fig:transmitted}]{ \includegraphics[width=3in]{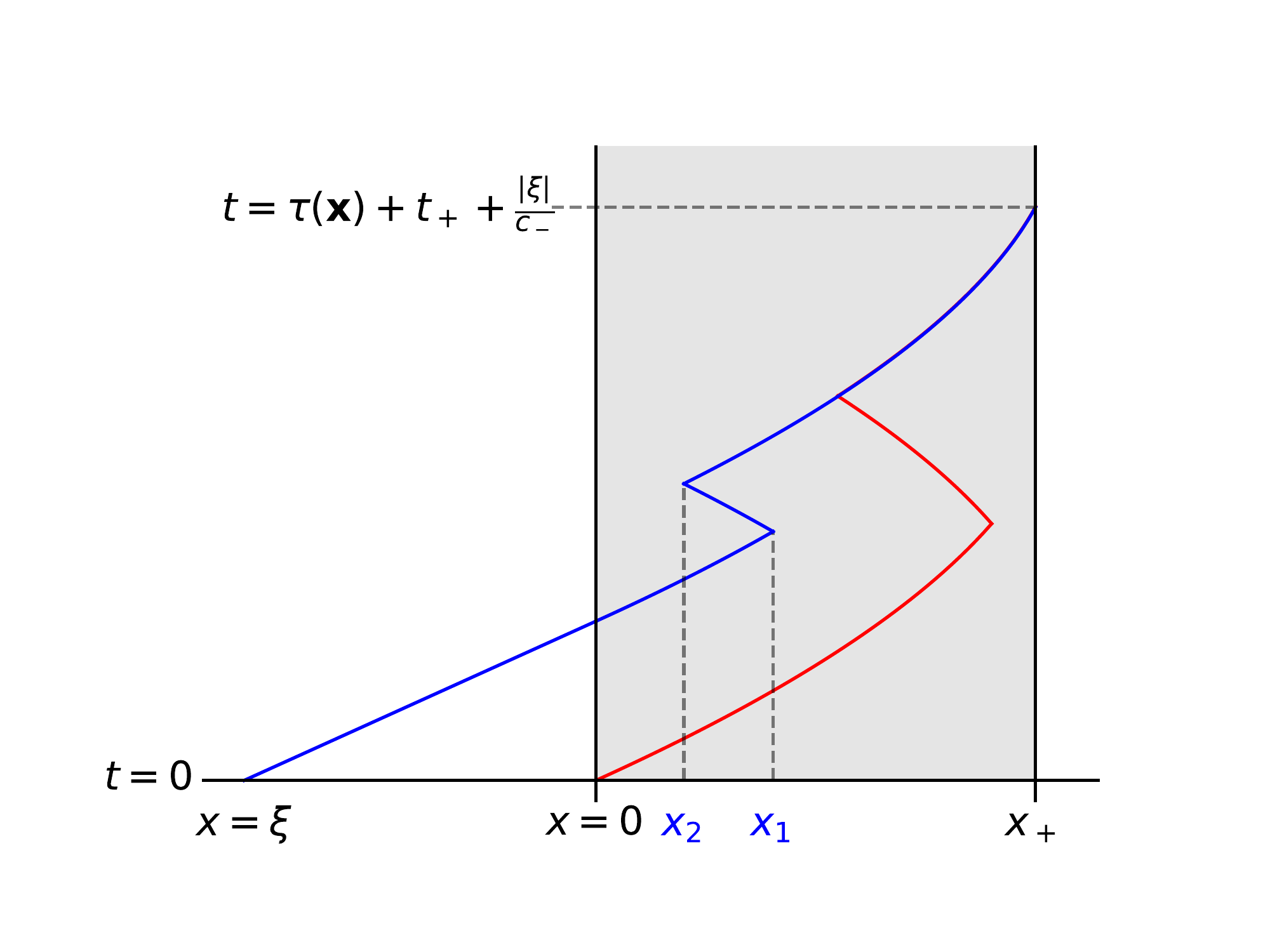}}
\subfloat[Characteristics contributing to $R_3(t)$\label{fig:reflected}]{ \includegraphics[width=3in]{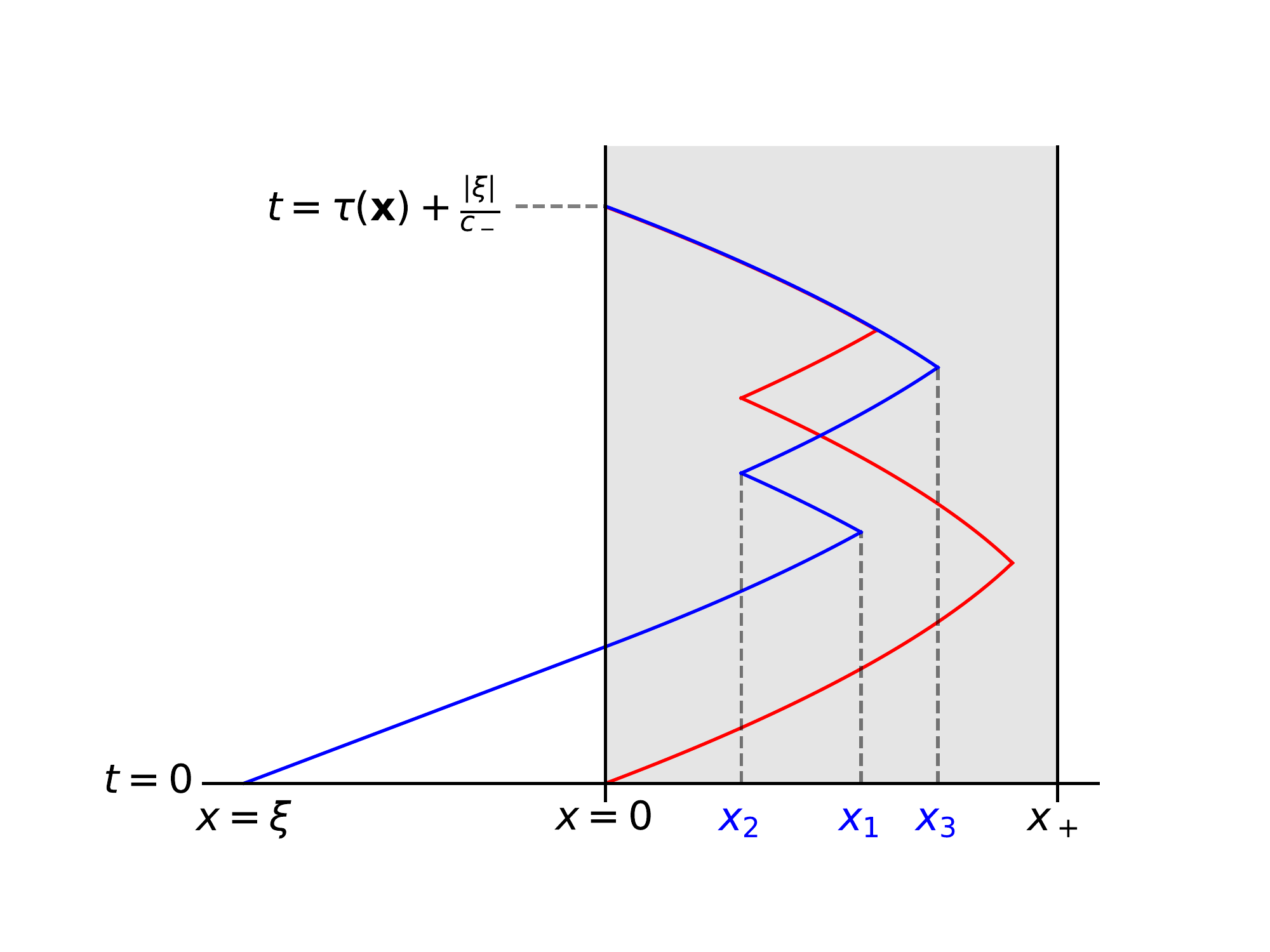}}
\caption{Characteristic paths starting from different points but arriving simultaneously
to contribute to the indicated transmission and reflection terms.  For clarity, only the reflection
points of the blue trajectories are marked. \label{fig:chars}
}
\end{centering}
\end{figure}
Let us work out the initial location $\xi(\bx,x,t)$ for a path
passing through $x=0$ (going to the right)
that is subsequently reflected at the points $x_1, x_2, \dots, x_n\in(0,x_\righ)$
and eventually arrives going to the left at $(x,t)$, for some $x\in[0,x_+]$ (see \cref{fig:reflected}).
It is convenient to define
\begin{align} \label{t-hat}
\traveltime(x_1, x_2, \dots, x_n) = \traveltime(\bx) = 2\sum_{j=1}^{n}(-1)^{j+1} \tau_j,
\end{align}
where again $\tau_j$ is the travel time from $x=0$ to $x_j$,
The travel time for the path is $\tau(\bx)-\tau(x)$.
Thus this path must have first passed through
$x=0$ at time $t-\traveltime(\bx)+\tau(x)$.
Hence it must have originated at time zero from
\begin{align} \label{xi_R}
\xi_R(\bx,x,t) & := -c_\lef(t-\traveltime(\bx)+\tau(x)) & t\ge \traveltime(\bx)-\tau(x).
\end{align}

We can compute the contribution of all paths that are
eventually reflected, for any initial condition $p_0$,
using \eqref{reflection-integral} with $\xi_R(\bx,0,t)$ given by \eqref{xi_R}.

\subsection{Transmission}
Consider a path starting at $x=0$ (going to the right)
that is reflected at the points $x_1, x_2, \dots, x_n\in(0,x_\righ)$
and arrives at $(x,t)$ going to the right, for some $x\in[0,x_+]$ (see \cref{fig:transmitted}).
The total time to traverse this
path is $\traveltime(\bx) + \tau(x)$. Hence it must have originated at
time zero from
\begin{align} \label{xi_T}
\xi_T(\bx,x,t) = -c_-(t-\traveltime(\bx)-\tau(x)).
\end{align}
Each path contributing to the transmitted wave starts at $x\le0$ (with impedance $Z_-$) and ends
at $x=x_\righ$ (with impedance $Z_\righ$), so the net change in the value of the solution
along this path due to Green's law is given by the factor $C_G$ defined in \eqref{green23}.
Hence the contribution to the solution is given by $C_G r(x_1) r(x_2) \cdots r(x_n)$,
leading to the integral \eqref{transmission-integral} for the total contribution of all
paths with $n$ reflections.

We can compute the contribution of all paths that are
eventually transmitted, for any initial condition $p_0$,
using \eqref{transmission-integral} with $\xi_T(x,x_+,t)$ given by \eqref{xi_T}.
\new{In the next two subsections we specialize the formula above for the important
cases of a step (Heaviside) function and a $\delta$-function.}

\subsection{Scattering of a step function}
In this section we apply the approach just outlined to the propagation of an
initial condition consisting of a step:
\begin{align} \label{step}
p_0(x) = \begin{cases} 1 & x \le 0 \\ 0 & x > 0. \end{cases}
\end{align}
Since the step function is the integral of a $\delta$-function, the
resulting solution gives the integral of the Green's function for the
problem, and can be used as a basis to obtain solutions for arbitrary
initial data.

Straightforward calculation shows that the required values of $\xi$
in this case are simply
\begin{align} \label{stepfun_t}
p_0(\xi_R(\bx,x,t)) = \begin{cases} 0 & t < \traveltime(\bx) \\  1 & t \ge \traveltime(\bx) \end{cases}
\end{align}
for reflected components and
\begin{align} \label{stepfun_t_trans}
p_0(\xi_T(\bx,x,t)) = \begin{cases} 0 & t < \traveltime(\bx) + t_\righ \\  1 & t \ge \traveltime(\bx) + t_\righ \end{cases}
\end{align}
for transmitted components, where $\traveltime$ is defined in \eqref{t-hat}.

The integrals \eqref{reflection-integral} and \eqref{transmission-integral} for
the reflected and transmitted components can thus be written
\begin{subequations}
\begin{align} \label{reflection-integral-step}
    R_{2m+1}(t) & = (-1)^m \int \cdots \int_{\Paths^{[0,x_+]}_{2m}}
    \int^{x_+}_{\hat{x}_{2m+1}(t;\bx)} \prod_{j=1}^{2m+1} r(x_j) dx_j, \\
\label{transmission-integral-step}
    T_{2m}(t) & = (-1)^m C_G \int \cdots \int_{\Paths^{[0,x_+]}_{2m-1}}
    \int_0^{\hat{x}_{2m}(t;\bx)} \prod_{j=1}^{2m} r(x_j) dx_j.
\end{align}
\end{subequations}
Here the limits of integration $\hat{x}_{2m+1}(t;\bx)$ and $\hat{x}_{2m}(t;\bx)$
impose the condition that the path must reach $x=x_\righ$ by time $t$:
\begin{align*}
    \hat{x}_{2m}(t;\bx)  & = X((\traveltime(\bx)-t-t_+)/2) \\
    \hat{x}_{2m+1}(t;\bx) & = X((t-\traveltime(\bx))/2).
\end{align*}

\new{
\subsection{Scattering of a delta function}
Next we consider initial data consisting of a $\delta$-function:
$$
p_0(x) = \delta(x).
$$
Since this is the distributional derivative of the step function \eqref{step},
we can obtain the solution by differentiating \eqref{transmission-integral-step},
which yields
\begin{subequations}
\begin{align} \label{reflection-integral-delta}
    R_{2m+1}(t) & = (-1)^m \int \cdots \int_{\Paths^{[0,x_+]}_{2m}}
    r(\hat{x}_{2m+1}) \prod_{j=1}^{2m} r(x_j) dx_j, \\
\label{transmission-integral-delta}
    T_{2m}(t) & = (-1)^m C_G \int \cdots \int_{\Paths^{[0,x_+]}_{2m-1}}
    r(\hat{x}_{2m}(t;\bx)) \prod_{j=1}^{2m-1} r(x_j) dx_j.
\end{align}
\end{subequations}
}

\section{Relation between Green's coefficient and the transmission/reflection coefficients}\label{Asymptoticzigzags}
Let us consider what happens for long times; let $T_{2m}^\infty = \lim_{t\to\infty}T_{2m}(t)$.
Then $\hat{x}_{2m}(t;\bx)=x_+$ (for all $\bx$) and it is straightforward but tedious to evaluate the multiple integral
\eqref{transmission-integral-step}; the result depends only on $C_G$ and $m$.
For each value of $m$, $T_{2m}^\infty = (-1)^m a_{2m} C_G (\log(C_G))^{2m}$,
where the constants $a_{2m}$ for $m=1,2,3,\dots$ are
\begin{align} \label{sec-coeffs}
1, ~1/2, ~5/24, ~61/720, ~277/8064, ~50521/3628800, ~540553/95800320, \dots
\end{align}
We now explain where this sequence comes from.

For $t \ge (n+1) t_\righ$, the integral \eqref{transmission-integral-step} for $T_{2m}$ is over all of $\Paths_n$.
It can be simplified using the substitution $y(x) = \log(Z(x))/2$.
Also let $y_+=\log(Z_+)/2, y_-=\log(Z_-)/2$.  For simplicity we assume that $Z(x)$ is
monotone increasing.
Then
\begin{align}
\begin{split} \label{simple-integral}
    \frac{T_{2m}^\infty}{C_G}  & = (-1)^m \int_{y_-}^{y_+}dy_1
                                 \int_{y_-}^{y_1}dy_2 \int_{y_2}^{y_+}dy_3 \cdots \int_{y_-}^{y_{2m-1}}dy_{2m}
                                 = \Vol(\Paths^{[y_-,y_+]}_{2m}).
\end{split}
\end{align}
Let $n=2m$; then (ignoring the sign for the moment)
this integral is the volume of some subset of the $n$-dimensional
hypercube $[y_-,y_+]^n$; namely, the volume of the set $\Paths^{[y_-,y_+]}_{2m}$
(see \eqref{paths-def}).  It is does not include the full hypercube
because the reflection points are required to be an alternating sequence
(this requirement is enforced by the limits of integration).
Notice that since $Z(x)$ is monotone increasing, this is equivalent to the condition
that the sequence ${y_1, y_2, \dots, y_n}$ be alternating.
The integral in \eqref{simple-integral} gives the volume
of the subset of the hypercube that satisfies this alternating
condition.  The volume of the whole hypercube is of course $(y_+-y_-)^n=(\log(C_G))^n$.

To determine the value of the integral \eqref{simple-integral}, let us
partition the hypercube into $n!$ equal parts, where each part is defined by
a particular ordering  of the $y_j$.  For instance, with $n=4$ we would write
$$
V_{ijkl} = \{ (y_1, y_2, y_3, y_4) : y_i < y_j < y_k < y_l \},
$$
where $(i,j,k,l)$ ranges over all permutations of $(1,2,3,4)$.
Each of the sets $V_{ijkl}$ must have the same volume since there is
nothing to distinguish a particular coordinate direction.  Thus each
has volume $(\log(C_G))^n/n!$.  The value of the integral \eqref{simple-integral}
is determined by how many of the $V_{ijkl}$ satisfy the alternating condition.
With $n=4$ there are 5 alternating sequences:
$$
(4,2,3,1), ~(4,1,3,2), ~(3,2,4,1), ~(3,1,4,2), ~(2,1,4,3),
$$
so the integral yields $(5/24)(\log(C_G))^4$.
In general, the number of alternating sequences of length $n$ is known as
the $n$th {\em Euler zigzag number} (or just {\em zigzag number}); for even $n$
these are also known as secant numbers
or simply zig numbers \cite{andre1881,zigzag_oeis}.  We have proved
\begin{lemma} \label{volume-lemma}
Let $\Paths_n$ be defined by \eqref{paths-def}.  Then
\begin{align}
    \Vol\left(\Paths_n^{[\alpha,\beta]}\right) = \frac{A_n}{n!}(\beta-\alpha)^n,
\end{align}
where $A_n$ is the $n$th zigzag number;
i.e., the number of alternating permutations of a sequence of length $n$.
\end{lemma}
An immediate consequence is
\begin{theorem}
Let $Z(x)$ be monotone and define
\begin{align} \label{bn}
    b_n(z) = \frac{A_n}{n!} z^n
\end{align}
Then the asymptotic contributions for the step are given by
\begin{subequations}
\begin{align}
    T_{n}^\infty & = C_G b_n(i\log(C_G)) = C_G \frac{A_n}{n!}(i\log(C_G))^n & \text{for $n$ even} \label{Tnterm} \\
    R_{n}^\infty & = i b_n(i\log(C_G)) = i \frac{A_n}{n!}(i\log(C_G))^n & \text{for $n$ odd}, \label{Rnterm}
\end{align}
\end{subequations}
where $i$ denotes the imaginary unit.
\end{theorem}

The name {\em zigzag} seems eminently appropriate for numbers that appear in the
context of Figure \ref{fig:chars}.  Nevertheless, it is worth noting that the original
meaning of the name was a reference to zigzags in the discrete setting and had nothing
to do with space or paths.
There are many recursive formulas for the zigzag numbers; in the course of this work we
rediscovered the following formula by evaluating the multiple integrals \eqref{transmission-integral-step}
\cite{zigzag_oeis}.
Let $a_n=A_n/n!$; then the $a_n$ are generated by setting $a_0=a_1=1$ and computing
\begin{align*}
a_{2m} & = \sum_{j=1}^m \frac{(-1)^{j-1}}{(2j)!}a_{2(m-j)} \\
a_{2m+1} & = \sum_{j=1}^m \frac{(-1)^{j-1}}{(2j-1)!}a_{2(m-j+1)}.
\end{align*}
These formulas recover the values \eqref{sec-coeffs} and the corresponding sequence
for the reflection terms.
We recall the following combinatorial result due to Andr\'e \cite{andre1881}:
\begin{theorem}[Andr\'e's Theorem]
Let $b_n(z)$ be defined by \eqref{bn}.  Then
\begin{align*}
&    \sum_{m=0}^\infty b_n(z) = \sec(z) + \tan(z).
\end{align*}
\end{theorem}
Comparison of this result with our series \eqref{Tnterm}-\eqref{Rnterm}
leads immediately to
\begin{align}
\sum_{n=1}^\infty T_n^\infty + \sum_{n=1}^\infty R_n^\infty = C_G \sec(i\log(C_G)) + i \tan(i\log(C_G)).
\end{align}
Further comparing with the expressions for the transmission and reflection
coefficients yields
\begin{corollary} \label{trequiv}
Let $e^{-\pi} < Z_+/Z_- < e^\pi$.  Then
\begin{subequations} \label{TandRsums}
\begin{align}
\sum_{m=0}^\infty T_{2m}^\infty & = \Ct(Z_+,Z_-) = C_G \sech(\log(C_G)), \label{Tsum} \\
\sum_{m=0}^\infty R_{2m+1}^\infty& = \Cr(Z_+,Z_-) = \tanh(\log(C_G)).
\end{align}
\end{subequations}
\end{corollary}
\begin{proof}
We prove the transmission coefficient part; the proof for the reflection
coefficient is similar.
From \eqref{Tnterm} we have
$$
    \sum_{m=0}^\infty T_{2m}^\infty = C_G \sum_{m=0}^\infty \frac{A_{2m}}{(2m)!} (i\log(C_G))^{2m} = C_G\sech(\log(C_G)).
$$
This is the Maclaurin
series for $\sech(z)$ with $z=\log(C_G)$; the sequence is convergent for $|z| < \pi/2$,
which is equivalent to the condition $e^{-\pi} < Z_+/Z_- < e^\pi$.
Meanwhile, we can express the transmission coefficient in terms of the
Green's coefficient as follows:
\begin{align*}
    \Ct(Z_+,Z_-) & = \frac{2Z_+}{Z_+ + Z_-}  = \frac{2 C_G^2}{C_G^2+1}.
\end{align*}
Substituting $z=-i\log(C_G)$
(so $C_G=e^{iz}$) we find
\begin{align*}
    \Ct(Z_+,Z_-) & = \frac{2e^{2iz}}{e^{2iz}+1} = e^{iz}\sec(z) = C_G \sec(-i\log(C_G)) = C_G \sech(\log(C_G)).
\end{align*}
\end{proof}
Corollary \ref{trequiv} gives simple expressions for the transmission and
reflection coefficients in terms of the Green's coefficient.  It also says that
if we add up all the long-time asymptotic
contributions from paths with any even number of reflections, we obtain the
same value given by the transmission coefficient.  Similarly, if we add up all
contributions from paths with any odd number of reflections,
we obtain the same value as the reflection coefficient.
Thus the asymptotic state near $x=0$ for the reflection of the
step is just the middle state resulting from the Riemann problem.
In fact, \cref{trequiv} could instead be proven directly, using
PDE-based arguments to show that the net effect of all terms
asymptotically depends only on $Z_+, Z_-$ and so must sum to the
traditional transmission and reflection coefficients.


\new{
\begin{remark}
    The connection between paths and the zigzag numbers can be seen also in
the following way.  The path with reflection points $\{x_1, \dots, x_n\}$
can be associated with a tree, essentially as is done in
\cite[Section 2.3]{gibson2014combinatorics}, where each node corresponds to
a reflection point.  For each admissible ordering of the reflection points,
the resulting tree (with nodes labeled in the order in which they are visited)
is an increasing 0-1-2 tree, and it can be shown that all increasing 0-1-2
trees of $n$ nodes correspond to some ordering of the reflection points.
The number of increasing 0-1-2 trees with $n$ nodes is known to be equal
to $A_n$ \cite{callan2009note}.
\end{remark}
}

\subsection{Convergence\label{sec:convergence}}
\new{
In this section we consider the convergence of the series \eqref{solution-series}.
The analysis above and the examples in Section \ref{sec:examples} provide
evidence that the series \eqref{solution-series} derived in the previous sections
approximates the solution of the initial boundary value problem.
On the other hand, in the long-time limit, the series may diverge
for large impedance ratios.  Here we show that the series
\eqref{solution-series} always converges for any finite time.
The main result is Theorem \ref{main-theorem}; a more specialized result is given
in Theorem \ref{strong-theorem} that has much stronger assumptions but also much smaller
error bounds.

Let us extend formulas \eqref{TandR-integrals}, which approximate
solution values at $x=0$ and $x=x_+$ respectively, to approximate the solution
also for $x\in[0,x_+]$.  The same reasoning used in the previous sections leads
to the more general formulas
\begin{subequations} \label{char-series-def}
\begin{align} \label{Rchar-series-def}
    w^1_{2m+1}(x,t) := (-1)^m C_G(x) \int \cdots \int_{\Paths^{[0,x_+]}_{2m}} \prod_{j=1}^{2m} r(x_j) dx_j \int_{\max(x,x_{2m})}^{x_\righ} p_0(\xi_R(\bx,x,t)) r(x_{2m+1}) dx_{2m+1} \\
    w^2_{2m}(x,t) := (-1)^m C_G(x) \int \cdots \int_{\Paths^{[0,x_+]}_{2m-1}} \prod_{j=1}^{2m-1} r(x_j) dx_j \int_0^{\min(x,x_{2m-1})} p_0(\xi_T(\bx,x,t)) r(x_{2m}) dx_{2m}. \label{Tchar-series-def}
\end{align}
\end{subequations}
where
$$
    C_G(x) = \sqrt{Z(x)/Z_-}.
$$
Note that
\begin{align*}
    w^1_{2m+1}(x=0,t) & = R_{2m+1}(t) &
    w^2_{2m}(x=x_+,t) & = T_{2m}.
\end{align*}
\begin{lemma} \label{lem:w1w2}
Let $p_0(x), Z(x) \in C^1$.  Let $w^1_{2m+1}, w^2_{2m}$ be defined by \eqref{char-series-def} for $m\ge 0$ and
define $w^1_{-1}(x,t)=0$.  Then their partial derivatives exist and satisfy
    \begin{subequations}
    \begin{align}
        (w^1_{2m+1})_t - c(x) (w^1_{2m+1})_x & = r(x)\left(w^1_{2m+1}-w^2_{2m}\right) \\
        (w^2_{2m})_t + c(x) (w^2_{2m})_x & = r(x)\left(w^2_{2m}-w^1_{2m-1}\right).
    \end{align}
    \end{subequations}
\end{lemma}
\begin{proof}
The proof is by direct computation.  We illustrate by taking $m=0$.  First,
note that $C_G'(x) = r(x) C_G(x)$ and $\tau'(x)=c(x)$.  Thus we have
\begin{align*}
    w^2_0(x,t) & = C_G(x) p_0(-c_-(t-\tau(x))) \\
    (w^2_0)_t & = -c_- C_G(x) p_0'(-c_-(t-\tau(x))) \\
    (w^2_0)_x & = r(x) w^2_0 + c_- c(x) C_G(x) p_0'(-c_-(t-\tau(x))),
\end{align*}
so that
\begin{align*}
    (w^2_0)_t + c(x) (w^2_0)_x & = r(x) w^2_0.
\end{align*}
Next we have
\begin{align*}
    w^1_1(x,t) & = C_G(x) \int_x^{x_+} r(x_1) p_0(-c_-(t-2\tau(x_1) + \tau(x)))dx_1 \\
    (w^1_1)_t & = -c_- C_G(x) \int_x^{x_+} r(x_1) p_0'(-c_-(t-2\tau(x_1) + \tau(x)))dx_1 \\
    (w^1_1)_x & = r(x) w^1_1 - r(x) w^2_0 + c_- c(x) C_G(x) \int_x^{x_+}r(x_1) p_0'(-c_-(t-2\tau(x_1)+\tau(x))) dx_1,
\end{align*}
so that
\begin{align*}
    (w^1_1)_t + c(x) (w^1_1)_x & = r(x) \left(w^1_1 - w^2_0\right).
\end{align*}
\end{proof}
We remark that this result extends in a natural way to more general initial data
by using the theory of distributions.

Let us define formally
\begin{subequations} \label{w1w2sums}
\begin{align}
    \wlim^1(x,t) & := \sum_{m=0}^\infty w^1_{2m+1}(x,t) \\
    \wlim^2(x,t) & := \sum_{m=0}^\infty w^2_{2m}(x,t).
\end{align}
\end{subequations}
Using Lemma \ref{lem:w1w2} and assuming for the moment that the sums converge,
we can formally write
\begin{align*}
    \wlim^1_t - c(x)\wlim^1_x & = r(x)\left(\wlim^1-\wlim^2\right) \\
    \wlim^2_t + c(x)\wlim^2_x & = r(x)\left(\wlim^2-\wlim^1\right).
\end{align*}
showing that the limiting functions $\wlim^1(x,t), \wlim^2(x,t)$ satisfy \eqref{advection-coupled}
and therefore give the solution to the wave equation \eqref{acoustics}.
It remains to show that $\wlim^1, \wlim^2$ exist.
To this end, we require the following Lemma whose proof is deferred to the appendix.
\begin{lemma} \label{volume-lemma-2}
Consider the set
\begin{align} \label{paths-t-def}
    \Paths_n^{[0,x_\righ]}(t) := \{\bx \in [0,x_\righ]^n : \bx \text{ is an alternating sequence and }2\sum_j(-1)^{j+1}\tau_j \le t.\}.
\end{align}
Then
$$
    \Vol(\Paths_n(t)^{[0,x_\righ]}) \le \frac{t^n (\max_x c(x))^n}{n!}.
$$
\end{lemma}

\begin{lemma} \label{uniform-lemma}
Consider problem \eqref{acoustics} with $C^1$ coefficients \eqref{coefficients} and
initial data \eqref{initial-data}.  Let $w^{1,2}_n$ be defined as in \eqref{char-series-def}
and let $t<\infty$ be fixed.  Then the sums \eqref{w1w2sums} are uniformly
convergent, as are the sums
\begin{align}
    & \sum_{m=0}^\infty \frac{\partial}{\partial x} w^1_{2m+1}(x,t) & & \sum_{m=0}^\infty \frac{\partial}{\partial t} w^1_{2m+1}(x,t), \\
    & \sum_{m=0}^\infty \frac{\partial}{\partial x} w^2_{2m}(x,t) & & \sum_{m=0}^\infty \frac{\partial}{\partial t} w^2_{2m}(x,t).
\end{align}
Furthermore, we have
\begin{align*}
    \left|\sum_{n=N+1}^\infty w^1_{2m+1}(x,t)\right| & \le M C_G(x) \frac{(\zeta C(t+\tau(x)))^{2N+2}}{(2N+2)!}\sinh(\zeta Ct_*) \\
    \left|\sum_{n=N+1}^\infty w^2_{2m}(x,t)\right| & \le M C_G(x) \frac{(\zeta C(t-\tau(x)))^{2N+2}}{(2N+2)!}\cosh(\zeta Ct_*)
\end{align*}
for some $t_*\in[0,t]$, where
\begin{align*}
    C & = \max_x |c(x)|, &
    \zeta & = \max_x \frac{|Z'(x)|}{2Z(x)}, &
    M & = \max_x |p_0(x)|.
\end{align*}
\end{lemma}
\begin{proof}
We bound the magnitude of each term of each series by the product of the volume of integration
(using Lemma \ref{volume-lemma-2}) and the maximum magnitude of the integrand.
Notice that in the limits of integration for \eqref{Rchar-series-def} we can replace $\Paths_n$ by $\Paths_n(t+\tau(x))$ since
for paths outside the latter set we have $p_0(\xi_R(\bx,x,t))=0$.  Thus
\begin{align*}
    |w^1_{2m+1}(x,t)| & = C_G(x) \left|\int \cdots \int_{\Paths^{[0,x_+]}_{2m}(t+\tau(x))} \prod_{j=1}^{2m} r(x_j) dx_j \int_{\max(x,x_{2m})}^{x_\righ} p_0(\xi_R(\bx,x,t)) r(x_{2m+1}) dx_{2m+1}\right| \\
    & \le C_G(x) \int \cdots \int_{\Paths^{[0,x_+]}_{2m+1}(t+\tau(x))} |p_0(\xi_R(\bx,x,t))| \prod_{j=1}^{2m+1} |r(x_j)| dx_j \\
    & \le M C_G(x) \frac{(\zeta C (t+\tau(x))^{2m+1}}{(2m+1)!}.
\end{align*}
Therefore
\begin{align*}
    \sum_{m=0}^\infty |w^1_{2m+1}(x,t)| \le \sum_{m=0}^\infty M C_G(x) \frac{(\zeta C (t+\tau(x))^{2m+1}}{(2m+1)!} = M C_G(x) \sinh(\zeta C (t+\tau(x)).
\end{align*}
Similarly, we obtain
\begin{align*}
    \sum_{m=0}^\infty |w^2_{2m}(x,t)| \le M C_G(x) \cosh(\zeta C (t-\tau(x))) \\
    \sum_{m=0}^\infty \left|\frac{\partial}{\partial t} w^1_{2m+1}(x,t)\right| \le c_- D C_G(x) \sinh(\zeta Ct) \\
    \sum_{m=0}^\infty \left|\frac{\partial}{\partial t} w^2_{2m}(x,t)\right| \le c_- D C_G(x) \cosh(\zeta Ct) \\
    \sum_{m=0}^\infty \left|\frac{\partial}{\partial x} w^1_{2m+1}(x,t)\right| \le (|r(x)| M + c_- D) C_G(x) \sinh(\zeta Ct) \\
    \sum_{m=0}^\infty \left|\frac{\partial}{\partial x} w^2_{2m}(x,t)\right| \le (|r(x)|M + c_- D) C_G(x) \cosh(\zeta Ct),
\end{align*}
where $D = \max_x |p_0'(x)|$.
The error bounds in the theorem then follow from Taylor's theorem.
\end{proof}

\begin{remark}
Using \eqref{volume-lemma} without replacing $x_+$ in the arguments
above leads to estimates that are independent of $t$ but blow up when
$\max Z(x) / \min Z(x)$ is too large.
\end{remark}

Finally, we obtain
\begin{theorem} \label{main-theorem}
Consider problem \eqref{acoustics} with $C^1$ coefficients \eqref{coefficients} and
initial data \eqref{initial-data}.  Let $w^{1,2}_n$ be defined as in \eqref{char-series-def}
and let $t<\infty$ be fixed.  Let
$$
    \plim(x,t) = \sum_{m=0}^\infty w^1_{2m+1}(x,t) + \sum_{m=0}^\infty w^2_{2m}(x,t).
$$
Then $\plim(x,t)$ is the solution of the initial value problem.
\end{theorem}
\begin{proof}
Lemma \ref{uniform-lemma} shows that $\plim(x,t)$ is well-defined and also
(with Lemma \ref{lem:w1w2}) that $\wlim^{1,2}$ satisfy \eqref{advection-coupled}.
\end{proof}
}

The error estimate given in Lemma \ref{uniform-lemma} is typically too large to be
useful.  As we will see in the examples of Section \ref{sec:examples}, the series often converges much faster.
The next theorem gives an example of conditions under which more rapid convergence
can be guaranteed.

\begin{theorem} \label{strong-theorem}
Consider problem \eqref{acoustics} with coefficients \eqref{coefficients} and
unit step function initial data \eqref{step}.
Let $R_n(t)$ and $T_n(t)$ be defined as in \eqref{TandR-integrals}.
Let $Z(x)$ be monotone with $e^{-2\sqrt{2}} < Z_+/Z_- < e^{2\sqrt{2}}$.  Then
for any time $0 \le t < \infty$ the following limits exist:
\begin{align}
\lim_{N\to\infty} \sum_{m=1}^N R_{2m+1}(t) \\
\lim_{N\to\infty} \sum_{m=1}^N T_{2m}(t).
\end{align}
Furthermore, the terms $|R_{2m+1}(t)|$ and $|T_{2m}(t)|$ decrease monotonically with
$m$ and the approximation error can be bounded as follows:
\begin{align}
\left| \sum_{m=N}^\infty R_{2m+1}(t) \right| & \le
                |R_{2N+1}(t)| \le \left(\frac{C_G^2}{2}\right)^{N} |R_1(t)| \\
\left| \sum_{m=N}^\infty T_{2m}(t) \right| & \le |T_{2N}(t)| \le
                \left(\frac{C_G^2}{2}\right)^{N} |T_0(t)|.
\end{align}
\end{theorem}
Note that our assumption on the impedances gives $|\log(Z_+/Z_-)| < {2\sqrt{2}},$  and hence
we have $\frac{1}{2}C_G^2 < 1$ (and approaching $1/2$ as $Z_+/Z_- \rightarrow 1$), giving
exponentially fast convergence.
\begin{proof}
From \eqref{TandR-integrals} we see that if $Z(x)$ is monotone then the series $R_{2m+1}$
and $T_{2m}$ are alternating series (i.e., successive terms in each series have opposite sign).
It is sufficient to prove that the terms $|R_{2m+1}(t)|$ and $|T_{2m}(t)|$
decrease monotonically with $m$; then the rest of the theorem follows
from standard results for alternating series.
We prove convergence of the transmission series $T_{2m}(t)$.  The proof for the
reflection series is similar.
For simplicity, we consider the case in which $Z(x)$ is increasing.

Let $m$ and $t$ be fixed and let $Z(x)$ be as stated in the Theorem.
As discussed already $T_{2m}(t)$ is given by integrating over $\Paths_{2m}(t+t_+)$.
For clarity, in the remainder of the proof we write $\Paths_{2m}$ with no argument;
it is implicitly $t+t_+$.
\begin{align*}
    |T_{2m}(t)| & = \left|C_G \int \int \cdots \int_{\bx \in \Paths_{2m}} \prod_{j=1}^{2m} r(x_j) dx_j\right| \\
             & = |C_G| \int \int \cdots \int_{\bx \in \Paths_{2m}} \prod_{j=1}^{2m} |r(x_j)| dx_j.
\end{align*}
The second equality holds because, since
$Z(x)$ is monotone, the integrand has the same sign for all paths.
This also means that if $\Paths_{2m}$ is replaced by a larger set of paths, the resulting integral
provides an upper bound on $|T_{2m}(t)|$.

Notice that every path in $\Paths_{2m+2}$ can be obtained in exactly one
way by taking a particular path in $\Paths_{2m}$ and appending two (admissible)
reflection points $x_{2m+1}, x_{2m+2}$.  Admissibility of the resulting path involves
a restriction in the total path length (travel time $\traveltime(\bx)\le t$) and the condition that
$x_{2m+1} \ge \max(x_{2m}, x_{2m+2})$.  Let us consider the larger set $\widehat{\Paths}_{2m+2}$
obtained by omitting the path length restriction and requiring only that $x_{2m+1} \ge x_{2m+2}$.
In other words, $\widehat{\Paths}_{2m+2}$ is obtained by appending, for each path in $\Paths_{2m}$,
all pairs $(x_{2m+1}, x_{2m+2})$ such that $0 \le x_{2m+1} \le x_{2m+2} \le x_+$.
Clearly $\Paths_{2m+2}\subset\widehat{\Paths}_{2m+2}$, so we have
\begin{align*}
    |T_{2m+2}(t)| & = \left|C_G \int \int \cdots \int_{\bx \in {\Paths}_{2m+2}} \prod_{j=1}^{2m+2} r(x_j) dx_j\right| \\
    & \le \left|C_G \int \int \cdots \int_{\bx \in \widehat{\Paths}_{2m+2}} \prod_{j=1}^{2m+2} r(x_j) dx_j\right| \\
    & = |C_G| \int \int \cdots \int_{\bx \in \widehat{\Paths}_{2m+2}} \prod_{j=1}^{2m+2} |r(x_j)| dx_j \\
        & = |T_{2m}(t)| \int_{0}^{x_+} \int_{x_{m+2}}^{x_+} r(x_{2m+2}) r(x_{2m+1}) dx_{2m+1} dx_{2m+2} \\
        & = |T_{2m}(t)| \cdot \frac{1}{2} C_G^2.
\end{align*}

Since $|\log(Z_+/Z_-)| < {2\sqrt{2}},$ we have $\frac{1}{2}C_G^2 < 1$, so $|T_{2m+2}(t)| < |T_{2m}(t)|$,
so the alternating series is convergent.
\end{proof}

\subsection{Examples}\label{sec:examples}
In this section we illustrate, through numerical examples, the method just proposed.
\new{For comparison, we compute reference solutions using the finite volume solver Clawpack \cite{clawpack}.
Code for reproducing these results is available online.\footnote{\url{https://github.com/ketch/characteristics_rr}}
In each case the reference solution is computed with a discretization
sufficiently fine so that further
refinement produces no visible change in the solution.}

We take $x_+=1$ in all examples.
In the first three examples we take the functions $c(x), Z(x)$ to be linear in
the interval $(0,1)$:
\begin{align}
(c(x), Z(x)) & = \begin{cases} (c_-, Z_-) & x < 0 \\
                               ((1-x)c_- + x c_+, (1-x)Z_- + x Z_+) & 0\le x \le 1 \\
                               (c_+, Z_+) & x > 1. \end{cases}
\end{align}
Let $s = c_+ - c_-$.  Then
a right-going characteristic starting from $x=0$ at $t=0$ satisfies the ODE
\begin{align}
    X'(t) = c(x) & = (1-x)c_- + x c_+ & X(0)=0,
\end{align}
with solution
\begin{align}
    X(t) = \frac{c_-}{s}(e^{st}-1).
\end{align}
The total time to cross from $x=0$ to $x=x_+$ is thus
$$t_+ = \frac{1}{s}\log\left(\frac{s}{c_-} +1\right).$$

For each example, we show the solution corresponding to an initial step function
($p_0(x)=1$ for all $x<0$) \new{and a Dirac $\delta$-function ($p_0(x)=\delta(x)$)}.
For the $\delta$-function examples, the $\delta$-function part of the
transmitted wave is represented by a larger red circle that also indicates the mass
of the transmitted $\delta$-function.

A first example, with very mild variation in $Z$, is shown in \cref{fig:example1}.
The solution involving only terms up to $T_2$ is already highly accurate.
In the second example, shown in \cref{fig:example2}, $Z$ varies by a factor of 8.
In this case it can be seen that the approximation using terms up to $T_4$ gives a
significant improvement.

Both of the previous examples satisfy the conditions given in Theorem \ref{strong-theorem}.
The next two examples do not.
In the third example, we take $Z_-=1$ and $Z_+=20$.  It can be seen
that in this case the convergence for large times is much slower and the
series including terms up to $T_4$ is a good approximation only for short
times.

In the final example, $Z(x)$ is non-monotone:
$$
Z(x) = 0.25 + 0.75x + \sin(10\pi x)/10.
$$
The solution given by including terms up to $T_4$ captures the oscillating
solution well.  This example also illustrates that when $Z(x)$ is a non-monotone
function, the transmitted wave amplitude can exceed $C_G$ at some points.

\new{\begin{remark}
Although we have focused on media with continuous coefficients $Z(x), c(x)$,
it is possible to extend this approach to piece-wise continuous media
by incorporating the effect of reflection and transmission at points of
discontinuity. \cref{fig:piecewise} shows the approximation $R_1(t)$ for an
example with $x_+=1$ and a single discontinuity at $x=1/2$.
Let $Z^{\pm}_{1/2}$ denote the impedance just to the left and right of $x=1/2$,
and assume that $Z(x)$ varies continuously over each interval $[0,1/2]$ and $[1/2,1]$.
Then the transmission and reflection coefficients are given by
$C_T(Z^-_{1/2},Z^+_{1/2})$, $C_R(Z^-_{1/2},Z^+_{1/2})$ for waves incident
from the left and by
$C_T(Z^+_{1/2},Z^-_{1/2})$, $C_R(Z^+_{1/2},Z^-_{1/2})$ for waves incident
from the right (see \eqref{discontTR}).
We can write $R_1(t) = R_1^\textup{cont}(t) + R_1^\textup{discont}(t)$
where the contribution from the discontinuity is given by the reflection coefficient. The
continuous contribution, $R_1^\textup{cont}$, is obtained by integrating over the
two continuous portions.
For a step function initial condition, this is
\begin{align}
R_1^\textup{cont}(t) & = \int_0^{min(1/2,X(t/2))}r(x)dx
 + C_T(Z^-_{1/2},Z^+_{1/2}) C_T(Z^+_{1/2},Z^-_{1/2}) \int_{1/2}^{max(1/2,X(t/2))}r(x)dx
\end{align}
which can also be written as
\begin{align}
R_1^\textup{cont}(t) &=
  \begin{cases}
\frac{1}{2} \log\left(\frac{Z(X(t/2))}{Z_-}\right)
       & \text{if } t < 2\tau(1/2) \\
   \frac{1}{2} \log\left(\frac{Z^-_{1/2}}{Z_-}\right)  + \frac{1}{2} C_T(Z^-_{1/2},Z^+_{1/2}) C_T(Z^+_{1/2},Z^-_{1/2})\log\left(\frac{Z(X(t/2))}{Z^+_{1/2}}\right)       & \text{otherwise}
  \end{cases}
\end{align}
The term with the two transmission coefficient factors accounts for the
paths that pass through the discontinuity (once in each direction).
The number of separate integrals that must be evaluated increases for higher-order
terms and for media with more discontinuities.
\end{remark}}

\begin{figure}
\begin{centering}
\subfloat[Step function initial condition.]{ \includegraphics[width=3in]{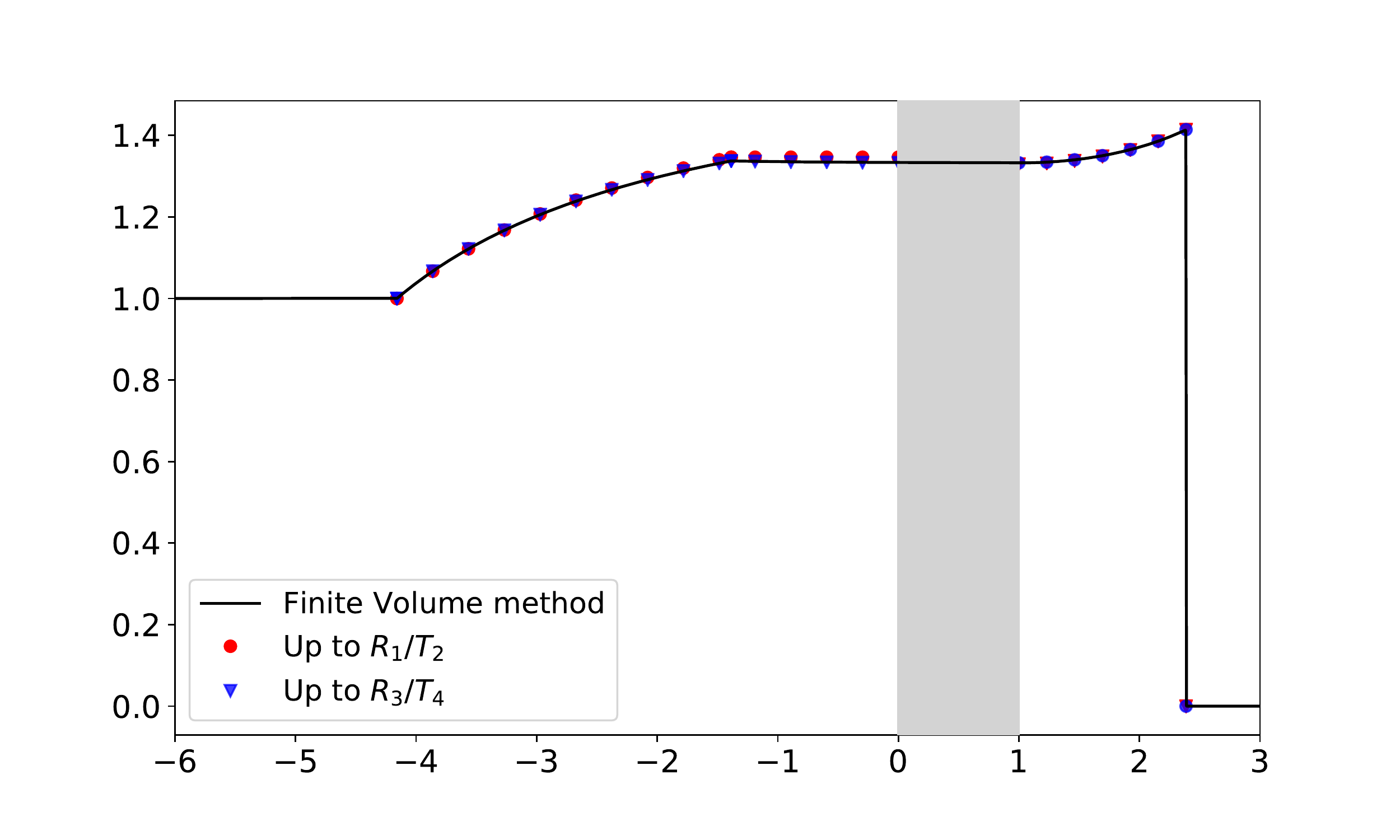}}
\subfloat[Dirac $\delta$-function initial condition.  Large red dot denotes location and mass of transmitted $\delta$-function.]{ \includegraphics[width=3in]{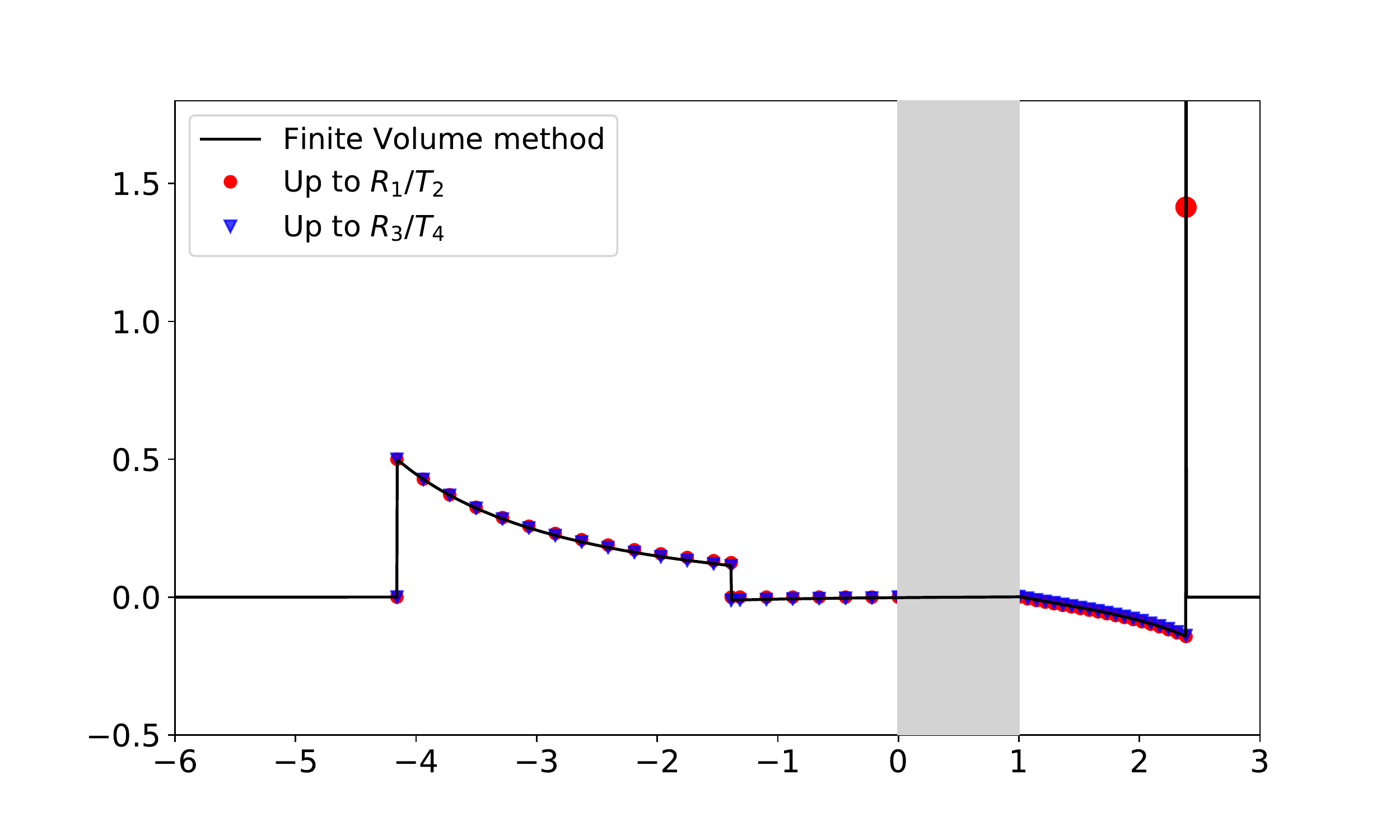}}
\caption{Solution at $t=3t_+$.  Here $x_+=1$, $c_-=2$, $c_+=1$, $Z_-=1/2$, and $Z_+=1$.
The solution is captured well by considering only two reflections.
\label{fig:example1}}
\end{centering}
\end{figure}

\begin{figure}
\begin{centering}
\subfloat[Step function initial condition.]{ \includegraphics[width=3in]{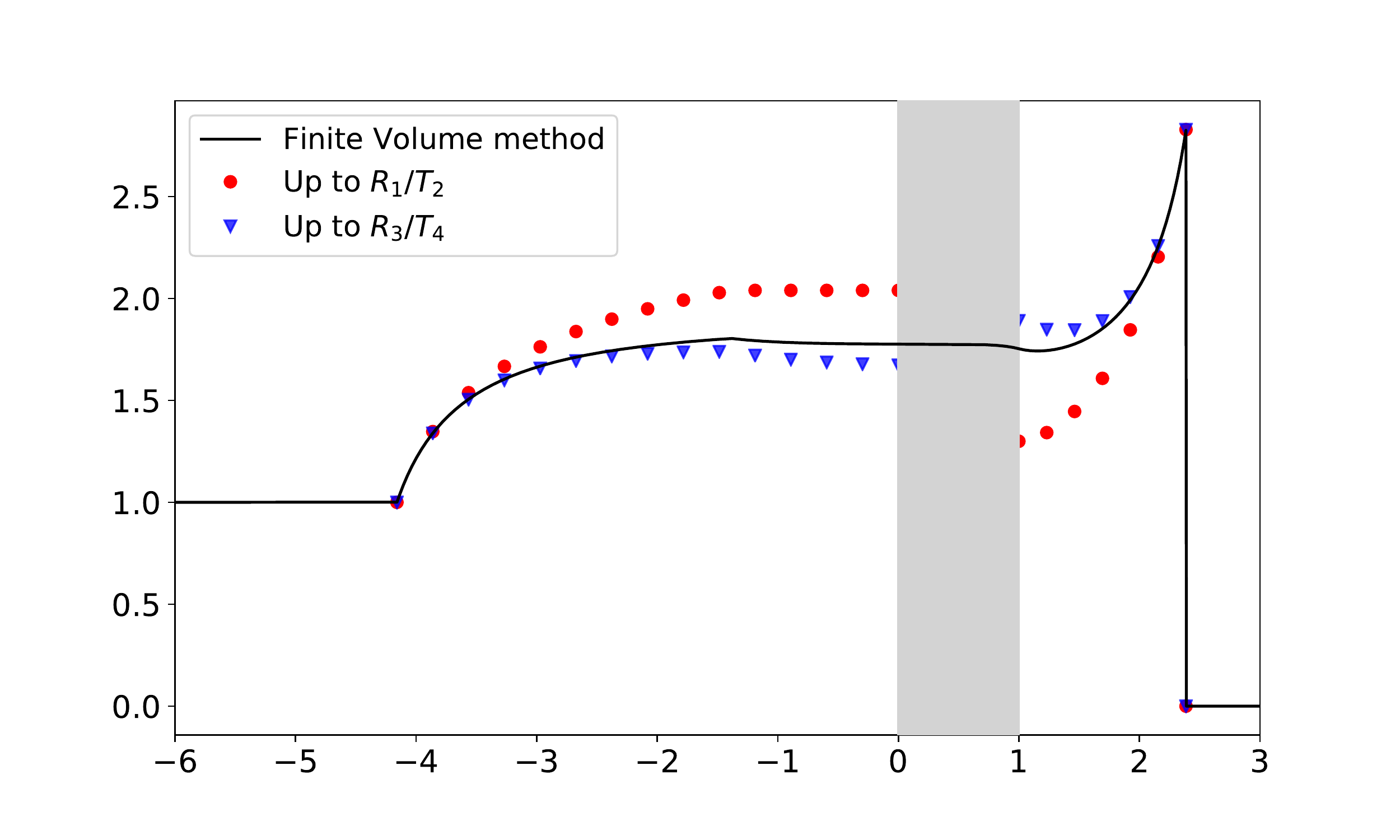}}
\subfloat[Dirac $\delta$-function initial condition.  Large red dot denotes location and mass of transmitted $\delta$-function.]{ \includegraphics[width=3in]{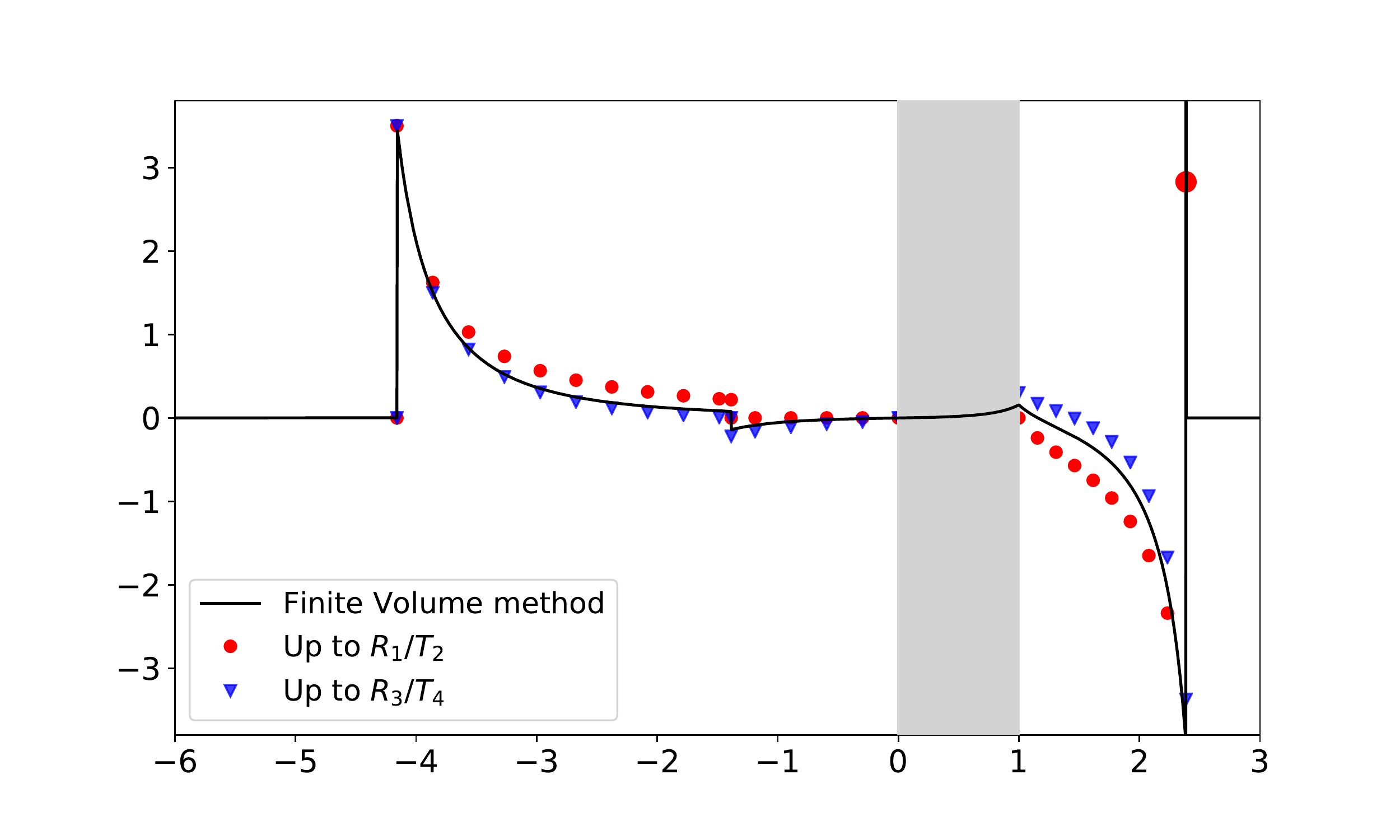}}
\caption{Solution at $t=3t_+$.  Here $x_+=1$, $c_-=2$, $c_+=1$, $Z_-=1/8$, and $Z_+=1$.
Using more reflections improves the accuracy of both the transmitted and reflected approximations.
\label{fig:example2}}
\end{centering}
\end{figure}

\begin{figure}
\begin{centering}
\subfloat[Step function initial condition.]{ \includegraphics[width=3in]{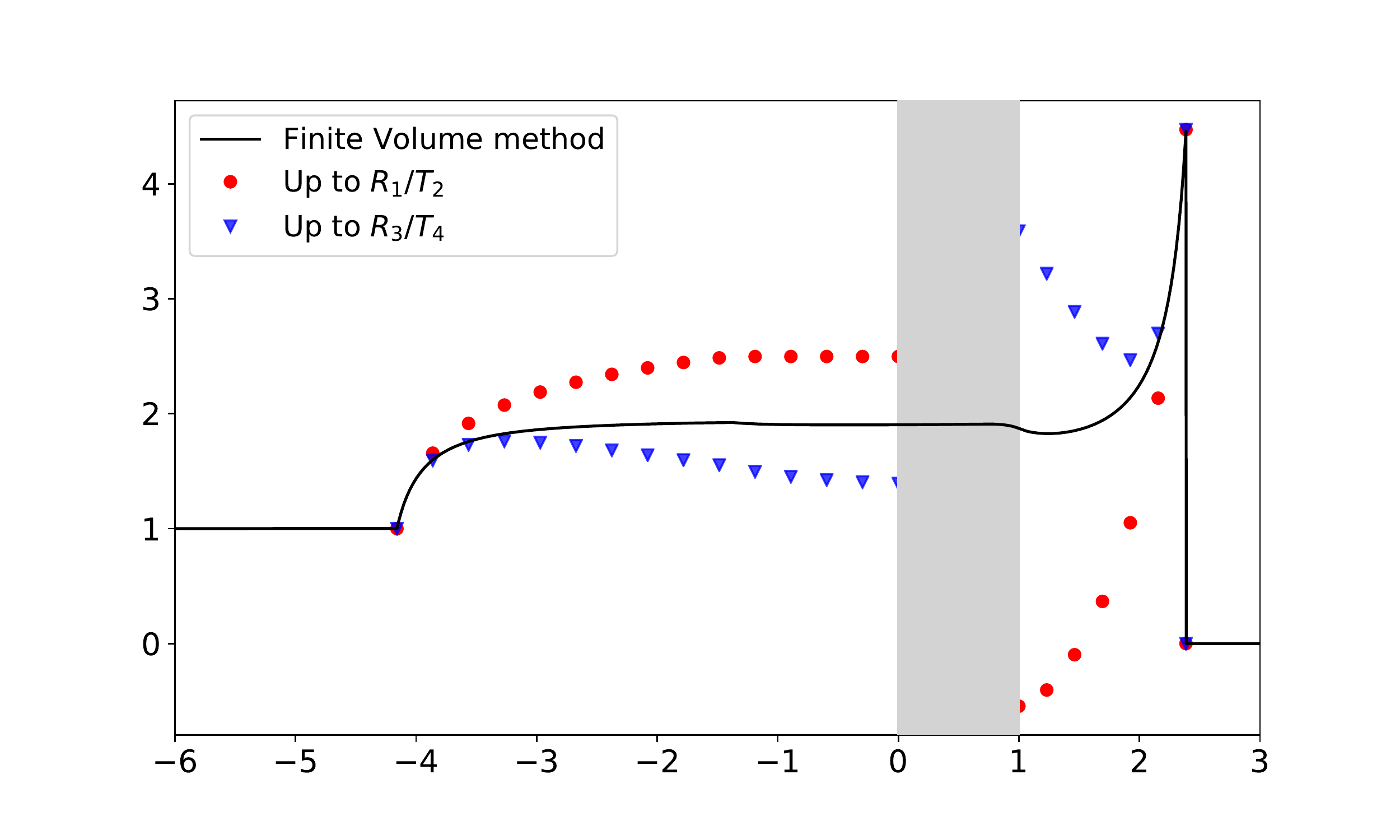}}
\subfloat[Dirac $\delta$-function initial condition.  Large red dot denotes location and mass of transmitted $\delta$-function.]{ \includegraphics[width=3in]{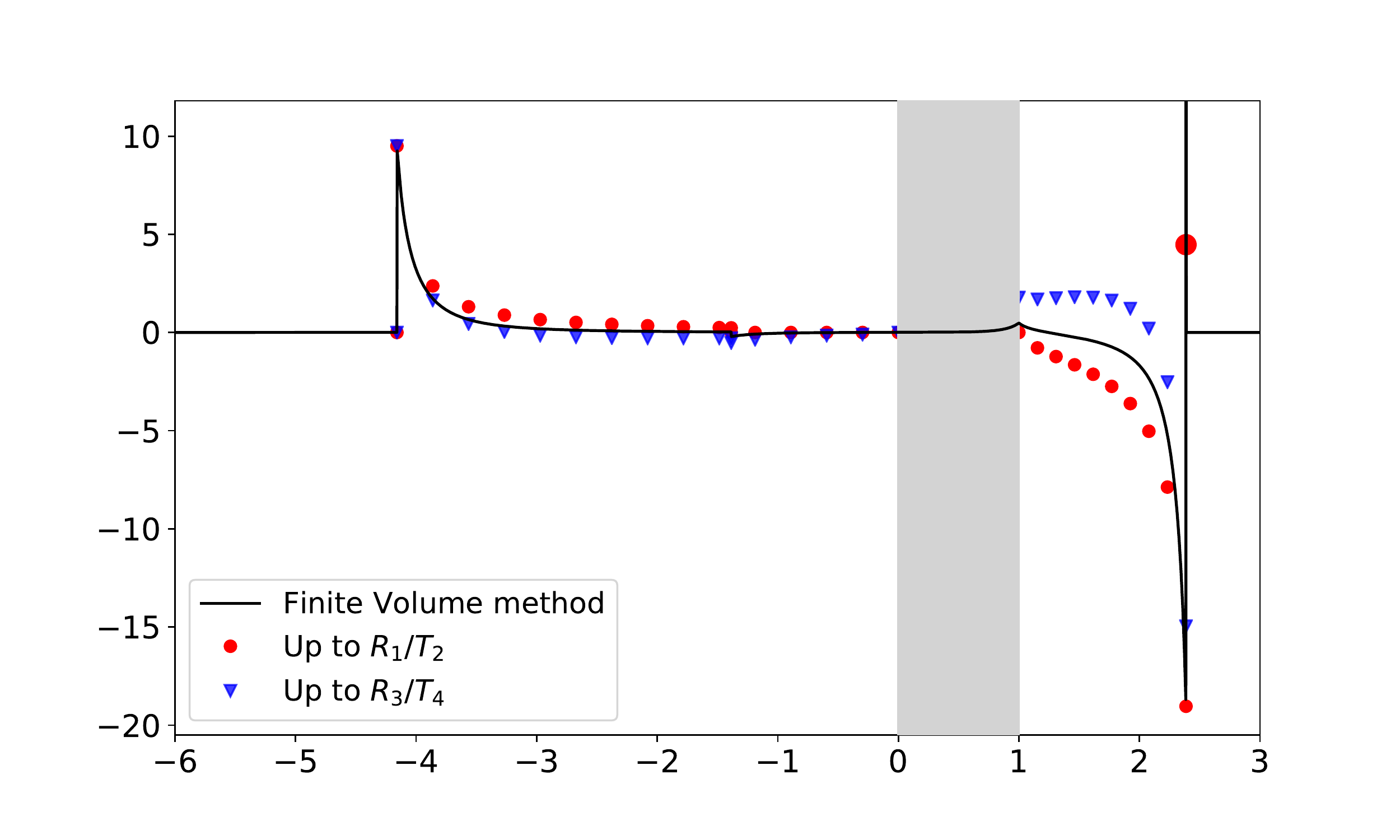}}
\caption{Solution at $t=3t_+$.  Here $x_+=1$, $c_-=2$, $c_+=1$, $Z_-=1$, and $Z_+=20$.
\label{fig:example3}}
\end{centering}
\end{figure}

\begin{figure}
\begin{centering}
\subfloat[Step function initial condition.]{ \includegraphics[width=3in]{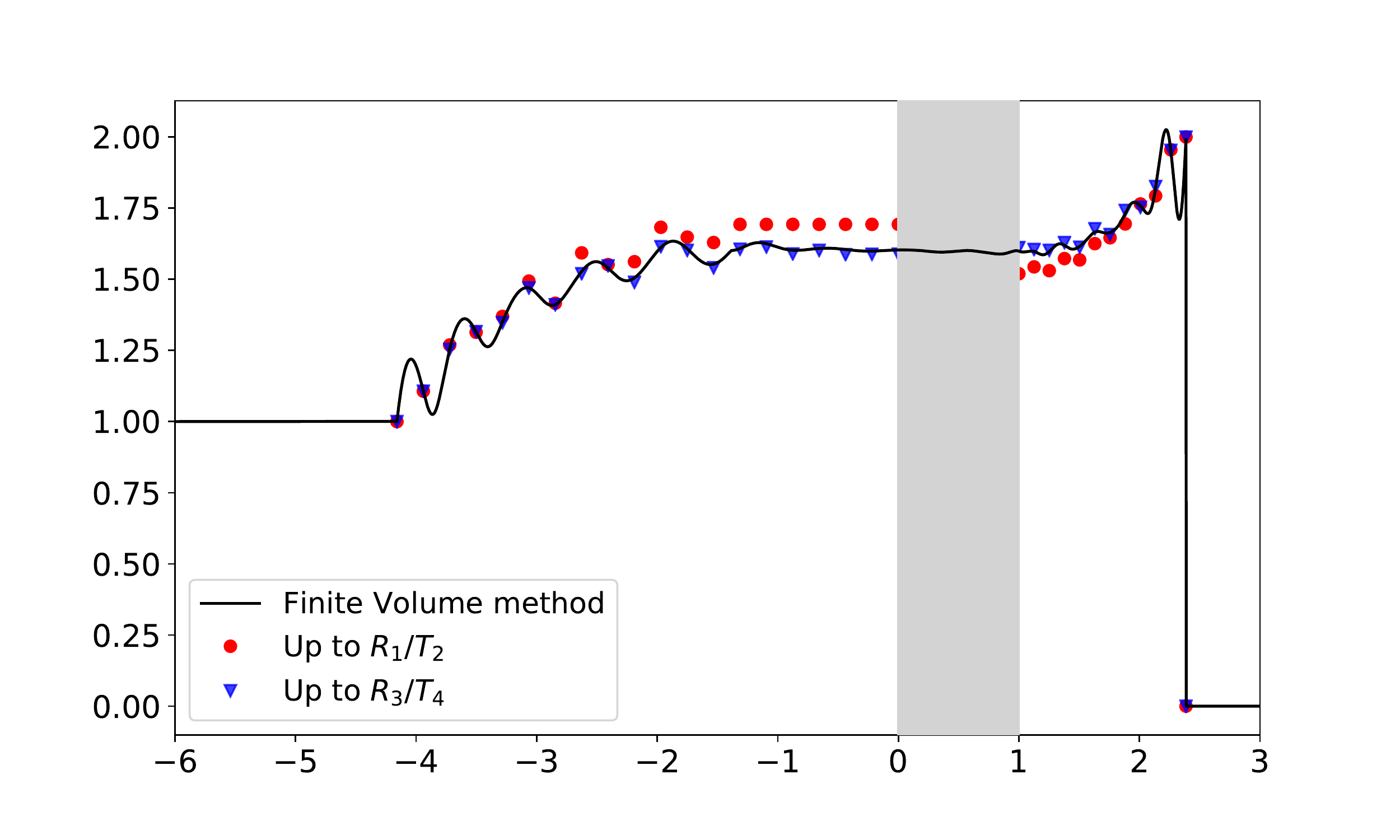}}
\subfloat[Dirac $\delta$-function initial condition.  Large red dot denotes location and mass of transmitted $\delta$-function.]{ \includegraphics[width=3in]{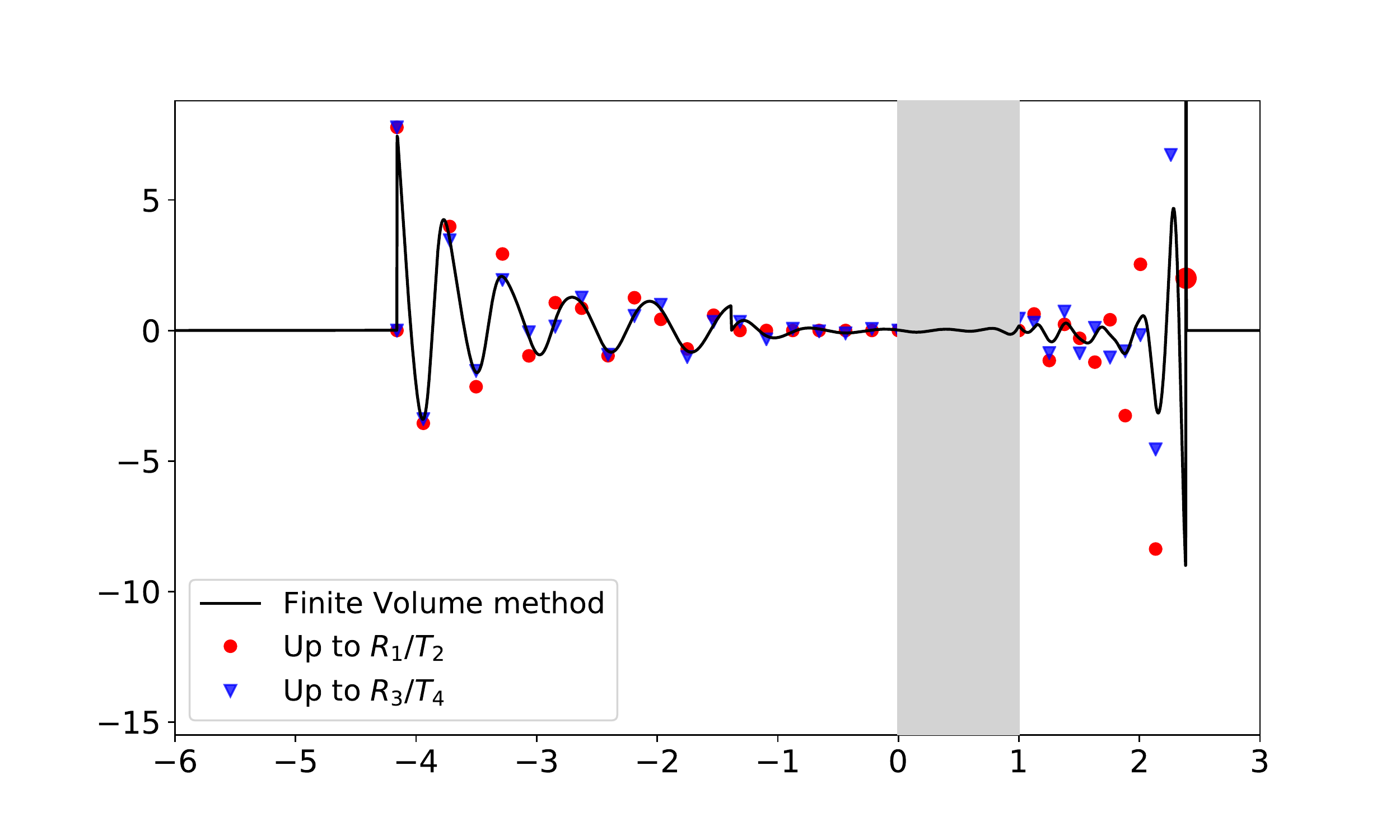}}
\caption{Solution at $t=3t_+$.  Here $x_+=1$, $c_-=2$, $c_+=1$, $Z_-=1$, and $Z_+=1/4$.
In the shaded region, $Z(x) = 0.25 + 0.75x + \sin(10\pi x)/10$.
\label{fig:example4}}
\end{centering}
\end{figure}

\begin{figure}
\begin{centering}
\includegraphics[width=4in]{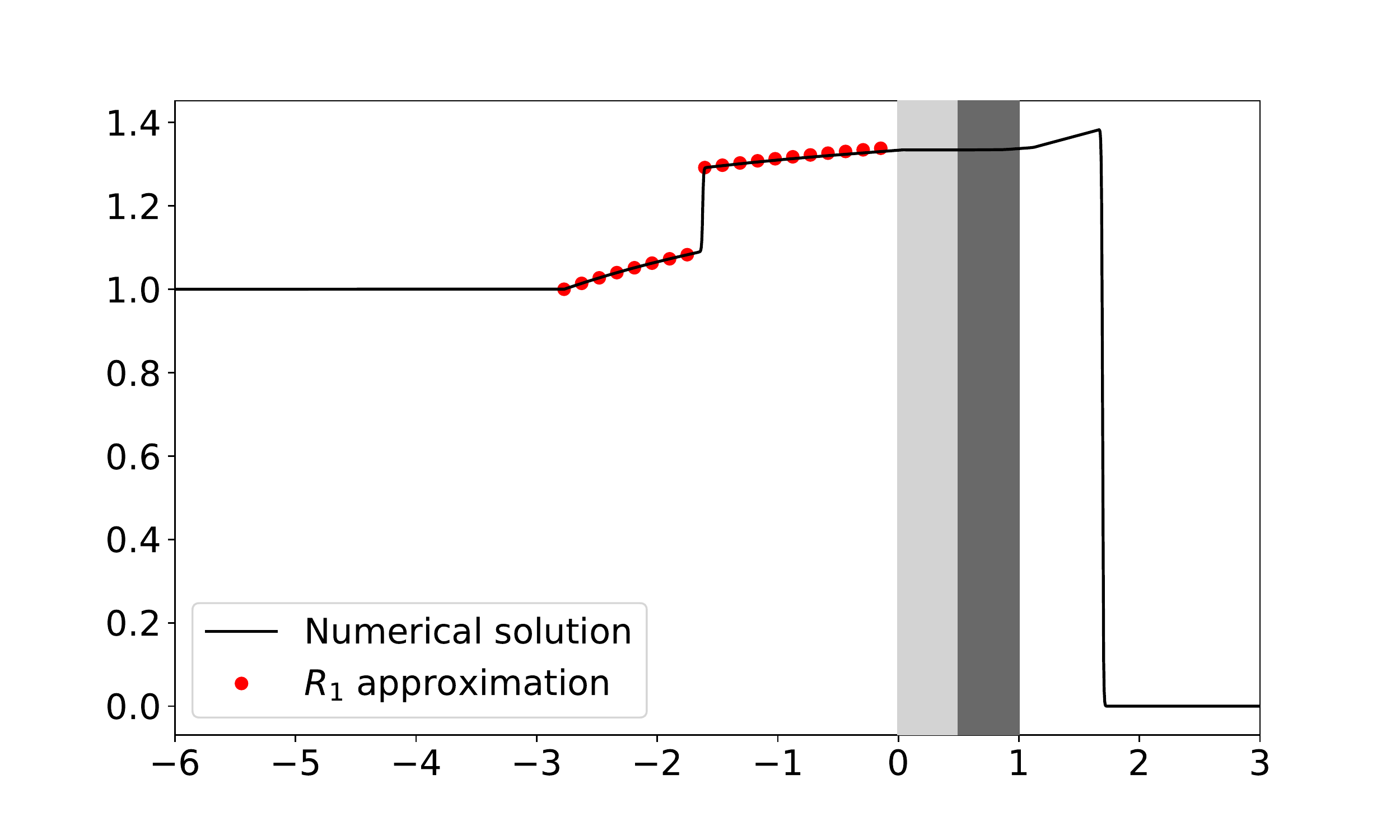}
\caption{Solution at $t=2t_r$ for a medium with piecewise-linear impedance
with $Z_- = 1$ at $x = 0$, $Z^-_{1/2} = 1.2$,
to $Z^+_{1/2} = 1.8$, and $Z_+ = 2$ at $x=1$.
\label{fig:piecewise}}
\end{centering}
\end{figure}

\section{Conclusions}
\new{We have developed a new approximation to the solution of the wave
equation in one space dimension in the
presence of a region of continuously-varying coefficients, by accounting
for all paths along which information can reach a given point.  This can be
extended in a straightforward way to other linear hyperbolic systems in
one dimension. This approximation gives an explicit expression for the
solution as an infinite sum of integrals depending only on the initial data
and the material properties.  We have shown that the series converges to the
solution of the wave equation.  We have also demonstrated that it can give
theoretical insight, by studying the propagation of a Heaviside function and
elucidating the relation between transmission and
reflection coefficients expected in the limiting case of a sharp interface,
and the Green’s law behavior expected for sufficiently smooth transitions in
material properties.

It is natural to ask how the approach described in this paper compares, as a
computational tool, to traditional numerical PDE discretizations like finite
difference or finite volume methods.  It is difficult to give a meaningful
answer to this question, because the approaches are fundamentally different:
\begin{itemize}
    \item Numerical methods begin with discretization: the medium is approximated,
        generally in a piecewise-constant manner and the solution is approximated
        by a representation in some finite basis.  The fundamental approximation in
        our approach is instead truncation of the series \eqref{solution-series};
	discretization is eventually required for numerical evaluation of
	integrals, but this can be done to machine precision if desired.
    \item Numerical discretizations require computation of the solution at a large
        number of points (in $(x,t)$, and this number must be increased in order to obtain
        higher accuracy.  In our approach the solution can be computed at a single
        point to any desired accuracy without computing the solution at other points.
    \item The importance of various factors influencing the size of the error are
        very different in numerical discretizations versus our approach.
        For instance, numerical discretizations have difficulty in accurately capturing
        narrow peaks such as the leading part of the transmitted or reflected waves above.
        In order to capture these, we had to use especially fine grids in Clawpack.
        But the path integral method is most accurate at these points; in fact,
        the first term in the infinite series already gives the exact solution.
\end{itemize}
Because of these differences, it is easy to construct situations in which one
approach or the other is vastly more efficient.  For instance, the method
described here can be more efficient if the solution is needed only at one or a
few points and if the initial data is not smooth.  On the other hand, if the
ratio $\max Z(x)/\min Z(x)$ is large and/or solution values are needed at very
many points, the approach described here may be much more costly than
traditional numerical discretizations.  We have not investigated techniques for
reducing the computational cost or made any detailed comparisons.
}

It is natural to expect that the series \eqref{solution-series} may converge because
paths involving many reflections contribute in successively smaller
amounts to the solution.  Examining \eqref{TandR-integrals}, this viewpoint makes sense only
if $|r(x)|<1$.  However, our examples and analysis show that
\eqref{solution-series} converges quite independently of any such condition.
Theorem \ref{strong-theorem} indicates that in general \eqref{solution-series} converges for a completely
different reason: the number of contributing paths (more precisely,
the volume they occupy in an appropriate space) becomes vanishingly small
as $n \to \infty$.

\section*{Acknowledgments}
We are grateful to Ernst Hairer for a comment that led us to the connection
with zigzag numbers, and to Lajos L\'oczi for reviewing an early draft of this work.
We also thank an anonymous referee for very helpful comments and suggestions.

\appendix

\section{Proof of Lemma \ref{volume-lemma-2}}
\begin{proof}
First, for simplicity take $c(x)=1$ so that the travel time between two
points is just the distance between them.
Because $c(x)=1$, in this case the set $\Paths_n(t)$ (defined in \eqref{paths-t-def}) is just
$$
\Xset_n(t) := \left\{ \bx \in[0,x_+]^n : \text{$\bx$ is an alternating
sequence and } 2\sum_j (-1)^{j+1}x_j \le t  \right\},
$$
the set of alternating sequences with path length at most $t$.
Define the mapping $f: \Real^n \to \Real^n$ by
$$
f_i(\bx) = \begin{cases}
			x_1 & i = 1 \\
			x_{i-1}-x_i & \text{ for $i$ even} \\
			x_i-x_{i-1} & \text{ for $i>1$ odd.}
\end{cases}
$$
This mapping can be represented by a lower-triangular matrix whose diagonal entries are $\pm1$,
so it preserves volume.  Note also that
$$
    \|f(\bx)\|_1 = \traveltime(\bx) - x_n,
$$
and for any alternating sequence $\bx\ge 0$ we have $f(\bx)\ge0$.
Let $\Bset_{n+}^1(t)$ denote the intersection of the $n$-dimensional $L_1$ ball
of radius $t$ with the positive orthant:
$$
 \Bset^1_{n+}(t) = \left\{x \in [0,\infty)^n : \|x\|_1 \le t\right\}.
$$
For any $\bx\in\Xset_n(t)$, we have $f(\bx)\in\Bset_{n+}^1(t)$, so
$$
  \Vol(\Paths_n(t))  =  \Vol(\Xset_n(t)) =  \Vol(f(\Xset_n(t))) \le
         \Vol(\Bset_{n+}^1(t)) = \frac{t^n}{n!}.
$$
The value of the last integral is a classical result due to Dirichlet
\cite[p. 168]{dirichlet1839}.

To extend the proof to arbitrary $c(x)$, let $C = \max_x |c(x)|$.
Then the length of a path emerging at time $t$
is no greater than $Ct$, so $\Paths_n(t) \subset \Xset_n(Ct)$.  Thus
\begin{align*}
\Vol(\Paths_n(t)) \le \Vol(\Xset_n(Ct)) \le \Vol(\Bset^1_{n+}(Ct)) = \frac{(Ct)^n}{n!}.
\end{align*}

\end{proof}

\nocite{HeronDzvonkovskaya2015}
\nocite{SynolakisGreenlaw1991}
\nocite{SimHuang2015}
\nocite{delRazoLeVeque2016}

\bibliographystyle{siamplain}
\bibliography{references}
\end{document}